\documentclass[10pt]{amsart}

\usepackage[colorlinks]{hyperref}
\usepackage{color,graphicx,shortvrb}
\usepackage[latin 1]{inputenc}

\usepackage[active]{srcltx} 

\usepackage{enumerate}
\usepackage{amssymb}

\newtheorem{theorem}{Theorem}[section]

\newtheorem{corollary}[theorem]{Corollary}

\newtheorem{lemma}[theorem]{Lemma}

\newtheorem{proposition}[theorem]{Proposition}
\newtheorem{remark}[theorem]{Remark}

\def\11{\textbf{$1$}}

\DeclareGraphicsExtensions{.jpg,.pdf,.png,.eps}

\begin{document}

\title{Tingley's problem for spaces of trace class operators}

\author[F.J. Fern\'{a}ndez-Polo]{Francisco J. Fern\'{a}ndez-Polo}
\address{Departamento de An{\'a}lisis Matem{\'a}tico, Facultad de
Ciencias, Universidad de Granada, 18071 Granada, Spain.}
\email{pacopolo@ugr.es}

\author[J.J. Garc{\'e}s]{Jorge J. Garc\'{e}s}
\address{Departamento de Matem{\'a}tica,
Centro de Ci\^{e}ncias F{\'i}sicas e Matem{\'a}ticas,
Universidade Federal de Santa Catarina, Brazil}
\email{jorge.garces@ufsc.br}

\author[A.M. Peralta]{Antonio M. Peralta}

\address{Departamento de An{\'a}lisis Matem{\'a}tico, Facultad de
Ciencias, Universidad de Granada, 18071 Granada, Spain.}
\email{aperalta@ugr.es}

\author[I. Villanueva]{Ignacio Villanueva}
\address{Departamento de An{\'a}lisis Matem{\'a}tico \\
Facultad de Matem{\'a}ticas \\ Universidad Complutense de Madrid \\
Madrid 28040\\Spain}
\email{ignaciov@mat.ucm.es}




\date{}

\begin{abstract} We prove that every surjective isometry between the unit spheres of two trace class spaces admits a unique extension to a surjective complex linear or conjugate linear isometry between the spaces. This provides a positive solution to Tingley's problem in a new class of operator algebras.
\end{abstract}

\maketitle
\thispagestyle{empty}

\section{Introduction}

In 1987, D. Tingley published a study on surjective isometries between the unit spheres of two finite dimensional Banach spaces, showing that any such mappings preserves antipodes points (see \cite{Ting1987}). A deep and difficult geometric problem has been named after Tingley's contribution. Namely, let $f: S(X)\to S(Y)$ be a surjective isometry between the unit spheres of two Banach spaces $X$ and $Y$. Is $f$ the restriction of a (unique) isometric real linear surjection $T : X\to Y$?  This problem remains unsolved even if we asume that $X$ and $Y$ are 2-dimensional spaces. Despite Tingley's problem still being open for general Banach spaces, positive solutions have been found for specific cases (see, for example, \cite{ChenDong2011,Di:p,Di:C,Di:8,Di:1,Ding07,Ding2002,Ding2009, Di,DiLi,FangWang06,FerPe17, FerPe17b, FerPe17c,KadMar2012,PeTan16,Ta:8,Ta:1,Ta:p,Tan2014, Tan2016, Tan2016-2, Tan2016preprint, Wang} and \cite{YangZhao2014}), and each particular case has required strategies and proofs which are more or less independent.\smallskip

For the purposes of this note, we recall that positive solutions to Tingley's problem include the following cases: $f: S(c_0)\to S(c_0)$ \cite{Di:C}, $f: S(\ell_1)\to S(\ell_1)$ \cite{Di:1}, $f: S(\ell_\infty)\to S(\ell_\infty)$ \cite{Di:8}, $f: S(K(H))\to S(K(H'))$, where $H$ and $H'$ are complex Hilbert spaces \cite{PeTan16}, and $f: S(B(H))\to S(B(H'))$ \cite{FerPe17b}. It is well known that the natural dualities $c_0^* = \ell_1$, and $\ell_1^*=\ell_\infty$ admit a non-commutative counterparts in the dualities $K(H)^*= C_1(H)$ and $C_1(H)^*=B(H)$, where $C_1(H)$ is the space of all trace class operators on $H$. So, there is a natural open question concerning Tingley's problem in the case of surjective isometries between the unit spheres of two trace class spaces. In this paper we explore this problem and we prove that every surjective isometry between the unit spheres of two trace class spaces admits a unique extension to a surjective complex linear or conjugate linear isometry between the spaces (see Theorem \ref{t Tingley for trace class infinite dimension}).\smallskip

The results are distributed in three main sections. In section \ref{sec:2} we establish new geometric properties of a surjective isometry $f:S(X)\to S(Y)$ in the case in which norm closed faces of the closed unit balls of $X$ and $Y$ are all norm-semi-exposed, and weak$^*$ closed faces of the closed unit balls of $X^*$ and $Y^*$ are all weak$^*$-semi-exposed (see Corollary \ref{c for spaces with property of semi-exposition for faces}). Applying techniques of geometry and linear algebra, in section \ref{sec:3} we present a positive answer to Tingley's problem for surjective isometries $f: S(C_1(H)) \to S(C_1(H'))$ when $H$ and $H'$ are finite dimensional (see Theorem \ref{t Tingley trace class finite dim}). The result in the finite dimensional case play a fundamental role in the proof of our main result.

\section{Facial stability for surjective isometries between the unit spheres of trace class spaces}\label{sec:2}

Let $H$ be a complex Hilbert space. We are interested in different subclasses of the space $K(H)$ of all compact operators on $H$. We briefly recall the basic terminology. For each compact operator $a$, the operator $a^* a$ lies in $K(H)$ and admits a unique square root $|a| = (a^* a)^{\frac12}$. The \emph{characteristic numbers} of the operator $a$ are precisely the eigenvalues of $|a|$ arranged in decreasing order and repeated according
to multiplicity. Since $|a|$ belongs to $K(H)$, only an at most countably number of its eigenvalues are greater than zero. According to the standard terminology, we usually write $\mu_n (a)$ for the $n$-th characteristic number of $a$. It is well known that $(\mu_n (a))_n\to 0$.\smallskip

The symbol $C_1=C_1(H)$ will stand for the space of trace class operators on $H$, that is, the set of all $a\in K(H)$ such that $$\|a\|_1 := \left(\sum_{n=1}^{\infty} |\mu_n (a)| \right) <\infty.$$ We set $\|a\|_{\infty} =\|a\|$, where the latter stands for the operator norm of $a$. The set $C_1$ is a two-sided ideal in the space $B(H)$ of all bounded linear operators on $H$, and $(C_1, \|.\|_1)$ is a Banach algebra. If tr$(.)$ denotes the usual trace on $B(H)$ and $a\in K(H)$, we know that $a\in C_1$ if, and only if, tr$(|a|)<\infty$ and $\|a\|_1 = \hbox{tr} (|a|)$.  It is further known that the predual of $B(H)$ and the dual of $K(H)$ both can be identified with $C_1(H)$ under the isometric linear mapping $a\mapsto  \varphi_a$, where $\varphi_a (x) := \hbox{tr} (a x)$ ($a\in C_1(H), x\in B(H)$). The dualities $K(H)^*= C_1(H)$ and $C_1(H)^*=B(H)$ can be regarded as a non-commutative version of the natural dualities between $c_0$, $\ell_1$ and $\ell_\infty$.\smallskip

It is known that every element $a$ in $C_1(H)$ can be written uniquely as a (possibly finite) sum \begin{equation}\label{eq spectral resolution in trace class} a= \sum_{n=1}^{\infty} \lambda_n \eta_n \otimes \xi_n,
 \end{equation} where $(\lambda_n)\subset \mathbb{R}_0^{+}$, $(\xi_n)$, $(\eta_n)$ are orthonormal systems in $H$, and $\displaystyle \|a\|_1=\sum_{n=1}^{\infty} \lambda_n$.\smallskip

Along the paper, we shall try to distinguish between $C_1(H)$ and $C_1(H)^*\equiv B(H)$, however the reader must be warned that we shall regard $K(H)$ and $C_1(H)$ inside $B(H)$. For example, when an element $a$ in $C_1(H)$ is regarded in the form given in \eqref{eq spectral resolution in trace class}, the element $$s(a) = \sum_{n=1}^{\infty} \eta_n \otimes \xi_n,$$ is a partial isometry in $B(H)$ (called the support partial isometry of $a$ in $B(H)$), and it is precisely the smallest partial isometry $e$ in $B(H)$ satisfying $e (a) = \|a\|_1$.\smallskip

We recall that two elements $a,b\in B(H)$ are orthogonal ($a\perp b$ in short) if and only if $a b^*= b^* a=0$. The relation ``being orthogonal'' produces a partial order $\leq$ in the set $\mathcal{U} (B(H))$ of all partial isometries given by $w\leq s$ if and only if $s-w$ is a partial isometry with $s-w \perp w$ (this is the standard order employed, for example, in \cite{AkPed92,EdRutt88}).\smallskip

We refer to \cite[Chapter III]{GohbergKrein}, \cite[\S 9]{DunSchw63}, \cite[Chapter II]{Tak} and \cite[\S 1.15]{S} for the basic results and references on the spaces $K(H)$, $C_1(H)$ and $B(H)$.\smallskip

It is worth recalling that, by a result due to B. Russo \cite{Ru69} every surjective complex linear isometry $T:C_1(H) \to C_1(H)$ is of the form $T(x) = v x u$ or of the form $T(x) = v x^t u$ ($x\in C_1(H)$), where $u$ and $v$ are unitary elements in $B(H)$ and $x^t$ denotes the transpose of $x$ (compare \cite[Theorem 11.2.2]{FleJam08}).\smallskip

The non-commutative Clarkson-McCarthy inequalities (see \cite[Theorem 2.7]{McCarthy67}) can be written as follows: \begin{equation}\label{Clarkson-McCarthy ineq 0pleq2}  \left( \|a\|_1 + \|b\|_1\right) \leq \|a+ b\|_1+ \|a-b\|_1 \leq 2 \left( \|a\|_1 + \|b\|_1 \right)
\end{equation} holds for every $a$ and $b$ in $C_1(H)$. It is further known that equality $$\|a+ b\|_1+ \|a-b\|_1 = 2 \left( \|a\|_1 + \|b\|_1 \right)$$ holds in \eqref{Clarkson-McCarthy ineq 0pleq2} if and only if $(a^* a) (b^* b)=0$, which is equivalent to say that $a$ and $b$ are orthogonal as elements in $B(H)$ ($a\perp b$ in short), or in other words $s(a)$ and $s(b)$ are orthogonal partial isometries in $B(H)$ (i.e. $(s(a)^* s(a), s(b)^* s(b))$ and $(s(a) s(a)^*, s(b) s(b)^*)$ are two pairs of orthogonal projections in $B(H)$). Consequently, if we fix $a,b\in S(C_1(H))$ we can conclude that \begin{equation}\label{eq orthogonality in Cp} \|a\pm b\|_1 = 2 \Leftrightarrow a\perp b \hbox{ (in $C_1(H)$)} \Leftrightarrow s(a)\perp s(b) \hbox{ (in $B(H)$).}
\end{equation}

Let us recall a technical result due to X.N. Fang, J.H. Wang and G.G. Ding, who established it in \cite{FangWang06} and \cite{Ding07}, respectively.

\begin{lemma}\label{l FaWang}{\rm(\cite[Corollary 2.2]{FangWang06}, \cite[Corollary 1]{Ding07})} Let $X$ and $Y$ be normed spaces and let $f : S(X) \to S(Y)$ be a surjective isometry. Then for any $x,y$ in $S(X)$, we have $\| x + y \|= 2$ if and only if $\| f(x) + f(y) \|= 2$.$\hfill\Box$
\end{lemma}

Throughout the paper, the extreme points of a convex set $C$ will be denoted by $\partial_e(C)$, and the symbol $\mathcal{B}_{X}$ will stand for the closed unit ball of a Banach space $X$. We shall write $\mathbb{T}$ for the unit sphere of $\mathbb{C}$. Following standard notation, the elements in $\partial_e(\mathcal{B}_{C_1(H)})$ are called \emph{pure atoms}. It is known that every pure atom in $C_1(H)$ is an operator of the form $\eta\otimes \xi$, where $\xi$ and $\eta$ are elements in $S(H)$.\smallskip

Given $\xi,\eta$ in a Hilbert space $H$, the symbol $\eta\otimes \xi$ will denote the rank one operator on $H$ defined by $\eta\otimes \xi (\zeta) = \langle\zeta | \xi \rangle \eta$ $(\zeta\in H$). Clearly $\eta\otimes \xi\in C_1(H)$. When $\eta\otimes \xi$ is regarded as an element in $C_1 (H)$, we shall identify it with the normal functional on $B(H)$ given by $\eta\otimes \xi (x) = \langle x(\eta) |\xi\rangle$ ($x\in B(H)$). As it is commonly assumed, given $\phi\in B(H)_*$ and $z\in B(H)$ we define $\phi z, z \phi\in B(H)_*$ by $(\phi z) (x) = \phi (z x)$ and $(z \phi) (x) = \phi (x z)$ ($x\in B(H)$). Accordingly with this notation, for $\eta\otimes \xi$ in $C_1 (H)=B(H)_*$, we have $(\eta\otimes \xi) z = \eta\otimes z^*(\xi)$ and $ z (\eta\otimes \xi) = z(\eta) \otimes \xi,$ for every $z\in B(H)$.\smallskip

We also recall an inequality established by J. Arazy in \cite[Proposition in page 48]{Ar}: For each projection $p$ in $B(H)$ and every $x\in C_1(H)$, we have \begin{equation}\label{ineq Arazy} \|x\|_1^2 \geq \|p x p \|_1^2 + \| p x (1-p)\|_1^2 + \|(1-p) x p\|_1^2 + \|(1-p) x (1-p)\|_1^2
\end{equation}

Suppose $\{\xi_i\}_{i\in I}$ is an orthonormal basis of $H$. The elements in the set $\{\xi_i\otimes \xi_i: i \in I \}$ are mutually orthogonal in $C_1 (H)$. Actually, the dimension of $H$ is precisely the cardinal of the biggest set of mutually orthogonal pure atoms in $C_1(H)$.\smallskip

We can state now a non-commutative version of \cite[Lemma 3]{Di:1}.

\begin{lemma}\label{l preservation of orthogonality} Let $H$ and $H'$ be complex Hilbert spaces, and let $f: S(C_1(H))\to S(C_1(H'))$ be a surjective isometry. Then $f$ preserves orthogonal elements in both directions, that is, $a\perp b$ in $S(C_1(H))$ if and only if $f(a)\perp f(b)$ in $S(C_1(H'))$.
\end{lemma}

\begin{proof} Take $a,b$ in $S(C_1(H))$. We have already commented that $a\perp b$ if and only if $\|a\pm b\|_1 = 2$ (compare \eqref{eq orthogonality in Cp}). Since $f$ is an isometry we deduce that $$\|f(a)-f(b)\|_1=2.$$

Lemma \ref{l FaWang} implies that $\|f(a) +f(b)\|_1=2$, and hence $\|f(a) \pm f(b)\|_1=2$, which assures that $f(a)\perp f(b)$.
\end{proof}

Among the ingredients and prerequisites needed in our arguments we highlight the following useful geometric result which is essentially due to L. Cheng and Y. Dong \cite[Lemma 5.1]{ChenDong2011} and R. Tanaka \cite{Tan2016} (see also \cite[Lemma 3.5]{Tan2014}, \cite[Lemmas 2.1 and 2.2]{Tan2016-2}).

\begin{proposition}\label{p faces ChengDong11}{\rm(}\cite[Lemma 5.1]{ChenDong2011}, \cite[Lemma 3.3]{Tan2016}, \cite[Lemma 3.5]{Tan2014}{\rm)} Let $f: S(X) \to S(Y)$ be a surjective isometry between the unit spheres of two Banach spaces, and let $\mathcal{M}$ be a convex subset of $S(X)$. Then $\mathcal{M}$ is a maximal proper face of $\mathcal{B}_X$ if and only if $f(\mathcal{M})$ is a maximal proper {\rm(}closed{\rm)} face of $\mathcal{B}_Y$.$\hfill\Box$
\end{proposition}

The previous result emphasizes the importance of a ``good description'' of the facial structure of a Banach space. A basic tool to understand the facial structure of the closed unit ball of a complex Banach space $X$ and that of the unit ball of its dual space is given by the ``facear'' and ``pre-facear'' operations. Following \cite{EdRutt88}, for each $F\subseteq \mathcal{B}_X$  and $G\subseteq \mathcal{B}_{X^*}$, we define
$$ F^{\prime} = \{a \in \mathcal{B}_{X^*}:a(x) = 1\,\, \forall x \in F\},\quad
G_{\prime} = \{x \in \mathcal{B}_X :a(x) = 1\,\, \forall a \in G\}.$$ Then, $F^{\prime}$ is a weak$^*$ closed face of $\mathcal{B}_{X^*}$ and $G_{\prime}$
is a norm closed face of $\mathcal{B}_X$. The subset $F$ is said to
be a \emph{norm-semi-exposed face} of $\mathcal{B}_X$ if $F=(F^{\prime})_{\prime}$, while the subset $G$ is called a \emph{weak$^*$-semi-exposed face} of $\mathcal{B}_{X^*}$ if $G =(G_{\prime})^{\prime}$. The mappings $F \mapsto F^{\prime}$
and $G \mapsto G_{\prime}$ are anti-order isomorphisms between the
complete lattices $\mathcal{S}_n(\mathcal{B}_X)$ of norm-semi-exposed faces
of $\mathcal{B}_X$ and $\mathcal{S}_{w^*}(\mathcal{B}_{X^*})$ of weak$^*$-semi-exposed
faces of $\mathcal{B}_{X^*}$ and are inverses of each other.\smallskip

Our next result is a generalization of the above Proposition \ref{p faces ChengDong11}.

\begin{proposition}\label{t ChengDong for general faces} Let $f: S(X) \to S(Y)$ be a surjective isometry between the unit spheres of two Banach spaces, and let $C$ be a convex subset of $S(X)$. Suppose that for every extreme point $\phi_0\in \partial_e(\mathcal{B}_{X^*})$, the set $\{\phi_0\}$ is a weak$^*$-semi-exposed face of $\mathcal{B}_{X^*}$. Then $C$ is a {norm-semi-exposed face} of $\mathcal{B}_X$ if and only if $f(C)$ is a {norm-semi-exposed face} of $\mathcal{B}_Y$.
\end{proposition}

\begin{proof} We begin with an observation. By Eidelheit's separation theorem \cite[Theorem 2.2.26]{Megg98}, every maximal proper face of $\mathcal{B}_{X}$ is a norm-semi-exposed face (compare \cite[Lemma 3.3]{Tan2016}).\smallskip

Suppose $C\in \mathcal{S}_n(\mathcal{B}_X)$. We set $$\Lambda_C^+:=\Big\{ \mathcal{M}: \mathcal{M} \hbox{ is a maximal proper face of $\mathcal{B}_X$ containing $C$} \Big\}.$$  Let us observe that $\bigcap_{\mathcal{M}\in \Lambda_C^+ } \mathcal{M}$ is a proper face of $\mathcal{B}_X$ containing $C$. Since a non-empty intersection of proper norm-semi-exposed faces of $\mathcal{B}_X$ (respectively, weak$^*$-semi-exposed face of $\mathcal{B}_{X^*}$) is a proper norm-semi-exposed face of $\mathcal{B}_X$ (respectively, a proper weak$^*$-semi-exposed face of $\mathcal{B}_{X^*}$), the set $\bigcap_{\mathcal{M}\in \Lambda_C^+ } \mathcal{M}$ is a proper {norm-semi-exposed face} of $\mathcal{B}_X$.\smallskip

We shall next show that \begin{equation}\label{eq semiexp face intersection of max proper faces containing it} C = \bigcap_{\mathcal{M}\in \Lambda_C^+ } \mathcal{M}.
\end{equation}

The inclusion $\subseteq$ is clear. To see the other inclusion we argue by contradiction, and thus we assume that $C\subsetneqq \bigcap_{\mathcal{M}\in \Lambda_C^+ } \mathcal{M}.$ Therefore, $C^{\prime} \supsetneqq \left( \bigcap_{\mathcal{M}\in \Lambda_C^+ } \mathcal{M}\right)^{\prime}$. The sets $C^{\prime}$ and $\left( \bigcap_{\mathcal{M}\in \Lambda_C^+ } \mathcal{M}\right)^{\prime}$ are weak$^*$-closed convex faces of $\mathcal{B}_{X^*}$, and hence, by the Krein-Milman theorem $C^{\prime} = \overline{\hbox{co}}^{w^*} (\partial_e(C^{\prime}))$. 
Therefore, having in mind that $C^{\prime}$ is a face, we can find an extreme point $\phi_0\in \partial_e(C^{\prime})\subset \partial_e(\mathcal{B}_{X^*})$ such that $\phi_0\notin \left( \bigcap_{\mathcal{M}\in \Lambda_C^+ } \mathcal{M}\right)^{\prime}$. Since, by hypothesis, $\{\phi_0\}$ is a weak$^*$-semi-exposed face of $\mathcal{B}_{X^*}$, we can easily check that $\mathcal{M}_0 =\{\phi_0\}_{\prime}$ is a maximal proper face of $\mathcal{B}_{X}$. Furthermore, $\phi_0 \in C^{\prime}$ implies that $C\subseteq \{\phi_0\}_{\prime}$ and hence $\mathcal{M}_0 \in  \Lambda_C^+$. Clearly $\phi_0\in \mathcal{M}_0^{\prime}\subseteq  \left( \bigcap_{\mathcal{M}\in \Lambda_C^+ } \mathcal{M}\right)^{\prime},$ which is impossible. We have thus proved \eqref{eq semiexp face intersection of max proper faces containing it}.\smallskip

Clearly, $f(C)\subseteq f(\mathcal{M})$ for every $\mathcal{M}\in \Lambda_C^+$, and thus $f(C) \subseteq  \bigcap_{\mathcal{M}\in \Lambda_C^+ } f\left(\mathcal{M} \right).$ Applying $f^{-1}$ and \eqref{eq semiexp face intersection of max proper faces containing it} we get $$ C \subseteq f^{-1}\left( \bigcap_{\mathcal{M}\in \Lambda_C^+ } f\left(\mathcal{M} \right)\right) \subseteq  \bigcap_{\mathcal{M}\in \Lambda_C^+ } f^{-1}\left(f\left(\mathcal{M} \right)\right)= \bigcap_{\mathcal{M}\in \Lambda_C^+ } \mathcal{M} =C.$$ Therefore, the identity $f(C) =  \bigcap_{\mathcal{M}\in \Lambda_C^+ } f\left(\mathcal{M} \right)$ follows from the bijectivity of $f$.\smallskip

Since, by Proposition \ref{p faces ChengDong11}, for each $\mathcal{M}\in \Lambda_C^+$, $f(\mathcal{M})$ is a maximal proper face of $\mathcal{B}_X$ and hence norm-semi-exposed, the set $f(C)$ coincides with a non-empty intersection of norm-semi-exposed faces, and hence $f(C)$ is a norm-semi-exposed face too.
\end{proof}

In certain classes of Banach spaces where norm closed faces are all norm-semi-exposed and weak$^*$ closed faces in the dual space are all weak$^*$-semi-exposed, the previous proposition becomes meaningful and guarantees the stability of the facial structure under surjective isometries of the unit spheres. For example, when $X$ is a C$^*$-algebra or a JB$^*$-triple, every proper norm closed face of $\mathcal{B}_{X}$ is norm-semi-exposed, and every weak$^*$ closed proper face of $\mathcal{B}_{X^*}$ is weak$^*$-semi-exposed (see \cite{AkPed92}, \cite{EdFerHosPe2010}, and \cite{FerPe10}). The same property holds when $X$ is the predual of a von Neumann algebra or the predual of a JBW$^*$-triple (see \cite{EdRutt88}). Suppose $X$ and $Y$ are Banach spaces satisfying the just commented property, and $f: S(X)\to S(Y)$ is a surjective isometry. Clearly $f$ maps proper norm closed faces of $\mathcal{B}_{X}$ to proper norm closed faces of $\mathcal{B}_{Y}$ and preserves the order given by the natural inclusion. In this particular setting, for each extreme point $e\in \mathcal{B}_{X}$, the set $\{e\}$ is a minimal norm-semi-exposed face of $\mathcal{B}_{X}$ and hence $\{f(e)\} = f(\{e\})$ must be a minimal norm closed face of $\mathcal{B}_{Y}$, and thus $f(e)\in \partial_e (\mathcal{B}_Y)$. All these facts are stated in the next corollary.

\begin{corollary}\label{c for spaces with property of semi-exposition for faces} Let $X$ and $Y$ be Banach spaces satisfying the following two properties \begin{enumerate}[$(1)$]\item Every norm closed face of $\mathcal{B}_X$ {\rm(}respectively, of $\mathcal{B}_Y${\rm)} is norm-semi-exposed;
\item Every weak$^*$ closed proper face of $\mathcal{B}_{X^*}$ {\rm(}respectively, of $\mathcal{B}_{Y^*}${\rm)} is weak$^*$-semi-exposed.
\end{enumerate} Let $f: S(X) \to S(Y)$ be a surjective isometry. The following statements hold:
\begin{enumerate}[$(a)$]\item Let $\mathcal{F}$ be a convex set in $S(X)$. Then $\mathcal{F}$ is a norm closed face of $\mathcal{B}_{X}$ if and only if $f(\mathcal{F})$ is a norm closed face of $\mathcal{B}_{Y}$;
\item Given $e\in S(X)$, we have that $e\in \partial_e (\mathcal{B}_X)$ if and only if $f(e)\in \partial_e (\mathcal{B}_Y)$.$\hfil\Box$
\end{enumerate}\end{corollary}

We have already commented that Corollary \ref{c for spaces with property of semi-exposition for faces} holds when $X$ and $Y$ are von Neumann algebras, or predual spaces of von Neumann algebras, or more generally, JBW$^*$-triples or predual spaces of JBW$^*$-triples. It is well known that $C_1(H)$ is the predual of $B(H)$.

\begin{proposition}\label{p minimal in arbitrary dimension} Let $f:S(C_1(H))\to S(C_1(H'))$ be a surjective isometry, where $H$ and $H'$ are complex Hilbert spaces. Then the following statements hold:\begin{enumerate}[$(a)$]\item A subset $\mathcal{F}\subset S(C_1(H))$ is a proper norm-closed face of $\mathcal{B}_{C_1(H)}$ if and only if $f(\mathcal{F})$ is.
\item $f$ maps $\partial_e(\mathcal{B}_{C_1(H)})$ into $\partial_e(\mathcal{B}_{C_1(H')})$;
\item dim$(H)=$dim$(H')$.
\item For each $e_0\in \partial_e(\mathcal{B}_{C_1(H)})$ we have $f(i e_0 ) = i f(e_0)$ or $f(i e_0 ) = -i f(e_0)$;
\item For each $e_0\in \partial_e(\mathcal{B}_{C_1(H)})$ if  $f(i e_0 ) = i f(e_0)$ {\rm(}respectively, $f(i e_0 ) = -i f(e_0)${\rm)} then  $f(\lambda e_0 ) = \lambda f(e_0)$ {\rm(}respectively, $f(\lambda e_0 ) = \overline{\lambda} f(e_0)${\rm)} for every $\lambda\in \mathbb{C}$ with $|\lambda|=1$.
\end{enumerate}

\end{proposition}

\begin{proof} $(a)$ and $(b)$ are consequences of Corollary \ref{c for spaces with property of semi-exposition for faces}. Having in mind that the dimension of $H$ is precisely the cardinal of the biggest set of mutually orthogonal pure atoms in $C_1(H)$, statement $(c)$ follows from $(b)$ and Lemma \ref{l preservation of orthogonality}.\smallskip

$(d)$ Let $e_0\in \partial_e(\mathcal{B}_{C_1(H)})$. Let us set $\{e_0\}^{\perp} := \{x\in C_1(H): x\perp e_0 \}$.  By Lemma \ref{l preservation of orthogonality} we have $f\left(\{e_0\}^{\perp}\cap S(C_1(H))\right) = \{f(e_0)\}^{\perp}\cap S(C_1(H'))$ and $$f\left(\mathbb{T} e_0\right) =f\left(\{e_0\}^{\perp\perp}\cap S(C_1(H))\right) = \{f(e_0)\}^{\perp\perp}\cap S(C_1(H'))= \mathbb{T} f(e_0).$$ Therefore $f(i e_0) = \mu f(e_0)$ for a suitable $\mu \in \mathbb{T}$. Since $$|1-\mu|=\|f(e_0) -f(i e_0)\|_1 = \|e_0 - i e_0\|_1 = |1-i| = \sqrt{2},$$ we deduce that $\mu\in \{\pm i\}$, as desired.\smallskip

$(e)$ Suppose $e_0\in \partial_e(\mathcal{B}_{C_1(H)})$ and $f(i e_0 ) = i f(e_0)$. Let $\lambda\in \mathbb{T}$. Arguing as above, we prove that $f(\lambda e_0) = \mu f(e_0)$ for a suitable $\mu \in \mathbb{T}$. The identities $$|1-\mu|=\|f(e_0) -f(\lambda e_0)\|_1 = \|e_0 - \lambda e_0\|_1 = |1-\lambda| ,$$ and $$|i-\mu|=\|i f(e_0) -f(\lambda e_0)\|_1 =\|f(i e_0) -f(\lambda e_0)\|_1 = \|i e_0 - \lambda e_0\|_1 = |i-\lambda|,$$ prove that $\mu =\lambda$.
\end{proof}

Whenever we have a surjective isometry $f: S(C_1(H))\to S(C_1(H'))$, we deduce from the above proposition that $H$ and $H'$ are isometrically isomorphic, we can therefore restrict our study to the case in which $H=H'$.\smallskip

We complete this section by recalling a result established by C.M. Edwards and G.T. Ruttimann in \cite{EdRutt88} (later rediscovered in \cite{AkPed92}). More concretely, as a consequence of the result proved by C.M. Edwards and G.T. R\"{u}ttimann in \cite[Theorem 5.3]{EdRutt88}, we know that every proper norm-closed face $\mathcal{F}$ of $\mathcal{B}_{C_1(H)}$ is of the form \begin{equation}\label{eq general form of a face} \mathcal{F} =\{w\}_{_{\prime}} =\{x\in C_1(H)=B(H)_* : \|x\|_1=1 = x(w) \},
\end{equation} for a unique partial isometry $w\in B(H)$. Furthermore the mapping $w\mapsto \{w\}_{_\prime}$ is an order preserving bijection between the lattices of all partial isometries in $B(H)$ and all norm closed faces of $\mathcal{B}_{C_1(H)}.$\smallskip

If $H$ is a finite dimensional complex Hilbert space, a maximal (or complete) partial isometry $w\in B(H)$ is precisely a unitary element. Therefore by the just commented result (\cite[Theorem 5.3]{EdRutt88}, see also \cite[Theorem 4.6]{AkPed92}) every maximal proper (norm-closed) face $\mathcal{M}$ of $\mathcal{B}_{C_1(H)}$ is of the form \begin{equation}\label{eq general form of a face finite dimension}\mathcal{M} =\{u\}_{_{\prime}} =\{x\in C_1(H)=B(H)_* : \|x\|_1=1 = x(u) \},\end{equation} for a unique unitary element $u\in B(H)$.

\section{Surjective isometries between the unit spheres of two finite dimensional trace class spaces}\label{sec:3}

In this section we present a positive solution to Tingley's problem for surjective isometries $f:S(C_1(H))\to S(C_1(H))$, in the case in which $H$ is a finite dimensional complex Hilbert space.\smallskip

Our next result is a first step towards a solution to Tingley's conjecture in $M_2(\mathbb{C})$ when the latter is equipped with the trace norm.

\begin{proposition}\label{p affine real combinations dim 2} Let $f: S(C_1(H))\to S(C_1(H))$ be a surjective isometry, where $H$ is a two dimensional complex Hilbert space. Suppose $e_1, e_2$ is a (maximal) set of mutually orthogonal pure atoms in $S(C_1(H))$ and $\lambda_1,\lambda_2$ are real numbers satisfying $|\lambda_1|+|\lambda_2|=1$. Then $$f(\lambda_1 e_1 + \lambda_2 e_2) = \lambda_1 f(e_1) + \lambda_2 f(e_2).$$
\end{proposition}

\begin{proof} Under these assumptions $B(H)$ is $M_2(\mathbb{C})$ with the spectral or operator norm, and $C_1(H)$ is $M_2(\mathbb{C})$ with the trace norm. We can assume the existence of orthonormal basis of $H$ $\{\eta_1,\eta_2\},$ $\{\xi_1,\xi_2\}$, $\{\widetilde{\eta}_1, \widetilde{\eta}_2\}$ and $\{\widetilde{\xi}_1, \widetilde{\xi}_2\}$ such that $e_j = \eta_j\otimes \xi_j$ and $f(\eta_j\otimes \xi_j) = \widetilde{\eta}_j\otimes \widetilde{\xi}_j$ for every $j=1,2$ (compare Proposition \ref{p minimal in arbitrary dimension}$(b)$ and Lemma \ref{l preservation of orthogonality}).\smallskip

Proposition \ref{p minimal in arbitrary dimension}($(d)$ and $(e)$) shows that the desired statement is true when $\lambda_1 \lambda_2=0$.\smallskip

To simplify the notation, let $u_1$, $u_2$ and $v_2^*$ be the unitaries in $B(H)$ mapping the basis $\{\eta_1,\eta_2\},$ $\{\widetilde{\eta}_1, \widetilde{\eta}_2\}$ and $\{\widetilde{\xi}_1, \widetilde{\xi}_2\}$ to the basis $\{{\xi}_1,\xi_2\},$ respectively. Let $T_1, T_2 : C_1(H)\to C_1(H)$ be the surjective complex linear isometries defined by $T_1 (x) = u_1 x$ and $T_2 (x) = u_2 x v_2$ ($x\in C_1(H)$). We set $g = T_2|_{S(C_1(H))} \circ f \circ T_1^{-1}|_{S(C_1(H))} : S(C_1(H)) \to S(C_1(H))$. Then $g$ is a surjective isometry satisfying $$g(\phi_j) = \phi_j, \hbox{  where } \phi_j=\xi_j\otimes \xi_j, \ \ \forall 1\leq j \leq 2.$$ We can chose a matricial representation (i.e. the representation on the basis $\{\xi_1,\xi_2\}$) such that $\phi_1 = \left(
                                                                                                                                  \begin{array}{cc}
                                                                                                                                    1 & 0 \\
                                                                                                                                    0 & 0 \\
                                                                                                                                  \end{array}
                                                                                                                                \right)$ and $\phi_2 = \left(
                                                                                                                                  \begin{array}{cc}
                                                                                                                                    0 & 0 \\
                                                                                                                                    0 & 1 \\
                                                                                                                                  \end{array}
                                                                                                                                \right)$.\smallskip

We recall that every maximal proper (norm-closed) face $\mathcal{M}$ of $\mathcal{B}_{C_1(H)}$ is of the form $$\mathcal{M} =\{u\}_{_{\prime}} =\{x\in C_1(H)=B(H)_* : \|x\|_1=1 = x(u) =1\},$$ for a unique unitary element $u\in B(H)$ (see \eqref{eq general form of a face finite dimension} or \cite[Theorem 5.3]{EdRutt88}).\smallskip

We note that $\phi_1, \phi_2\in \{1\}_{_{\prime}}$ and $\{1\}_{_{\prime}}$ is precisely the set of all normal states on $B(H)$.\smallskip

By Proposition \ref{p faces ChengDong11} we know that $g(\{1\}_{_{\prime}}) = \{u\}_{_{\prime}},$ for a unique unitary element $u\in B(H)$. Since $\phi_j = g(\phi_j)\in \{u\}_{_{\prime}}$ we can easily check that $u=1$.\smallskip

We shall first assume that $0\leq \lambda_j$ for every $j$. Since $\lambda_1+ \lambda_2=1$ we have $g(\lambda_1 \phi_1 + \lambda_2 \phi_2) \in g(\{1\}_{_{\prime}}) = \{1\}_{_{\prime}}$. Therefore, the element $g(\lambda_1 \phi_1 + \lambda_2 \phi_2)$ must be a positive matrix $\left(
\begin{array}{cc}
  t & c \\
   \overline{c} & 1-t \\
    \end{array}
     \right)$ with trace and trace norm equal to one (i.e., $0\leq t\leq 1$ and $\Big\| \left(
      \begin{array}{cc}
         t & c \\
          \overline{c} & 1-t \\
           \end{array} \right)\Big\|_1=1$.\smallskip

Since, by hypothesis, we have $$\left\| \left(
\begin{array}{cc}
 t & c \\
  \overline{c} & 1-t \\
     \end{array}
       \right) - \left(
         \begin{array}{cc}
             1 & 0 \\
             0 & 0\\
                 \end{array}
                    \right)
                      \right\|_1= \left\| g(\lambda_1 \phi_1 + \lambda_2 \phi_2) - g(\phi_1) \right\|_1$$ $$= \| \lambda_1 \phi_1 + \lambda_2 \phi_2 - \phi_1 \|_1 =  1 -\lambda_1 + \lambda_2 = 2 (1-\lambda_1).$$
It is easy to check that the eigenvalues of the matrix $\left(
 \begin{array}{cc}
    t-1 & c \\
            \overline{c} & 1-t \\
                 \end{array}
                    \right)^2$ are exactly $\{(1-t)^2 + |c|^2, (t-1)^2 + |c|^2\}\},$ and hence $\left\| \left(
                         \begin{array}{cc}
                                 t-1 & c \\
                                        \overline{c} & 1-t \\
                                              \end{array}
                                                   \right) \right\|_1=  2 \sqrt{(1-t)^2 + |c|^2}$. Therefore,
\begin{equation}\label{eq 1 06 12}  1-\lambda_1=  \sqrt{(1-t)^2 + |c|^2}.
\end{equation}
The equality $$2 \sqrt{t^2 + |c|^2}=\left\| \left(
 \begin{array}{cc}
   t & c \\
    \overline{c} & 1-t \\
    \end{array}
     \right) - \left(
     \begin{array}{cc}
     0 & 0 \\
     0 & 1 \\
    \end{array}
   \right)
   \right\|_1= \left\| g(\lambda_1 \phi_1 + \lambda_2 \phi_2) - g(\phi_2) \right\|_1$$ $$= \| \lambda_1 \phi_1 + \lambda_2 \phi_2 - \phi_2 \|_1 =  \lambda_1 +1- \lambda_2 =  2( 1-\lambda_2) = 2 \lambda_1,$$ implies \begin{equation}\label{eq 2 06 12} \lambda_1=  1-\lambda_2=  \sqrt{t^2 + |c|^2}.
                                                                                                                                \end{equation}
Combining \eqref{eq 1 06 12} and \eqref{eq 2 06 12} we get $\lambda_1 = t$ and $c=0$, which shows that
$$ g(\lambda_1 \phi_1 + \lambda_2 \phi_2) =  \left(
                                               \begin{array}{cc}
                                                 \lambda_1 & 0 \\
                                                 0 & \lambda_2 \\
                                               \end{array}
                                             \right) = \lambda_1 \phi_1 + \lambda_2 \phi_2 = \lambda_1 g(\phi_1) + \lambda_2 g(\phi_2),
$$ equivalently, since $g = T_2|_{S(C_1(H))} \circ f \circ T_1^{-1}|_{S(C_1(H))} $ we obtain \begin{equation}\label{eq 3 positive coef} f(\lambda_1 e_1 + \lambda_2 e_2) = \lambda_1 f(e_1) + \lambda_2 f(e_2),
\end{equation} for every $\lambda_1,\lambda_2\geq 0$, with $\lambda_1+\lambda_2=1$ and every set $\{e_1, e_2\}$ of mutually orthogonal rank one elements in $S(C_1(H))$.\smallskip

Finally, suppose $\lambda_1,\lambda_2\in \mathbb{R}$, with $|\lambda_1|+|\lambda_2|=1$. Set $\sigma_i \in \{\pm 1\}$ such that $\lambda_i = \sigma_i |\lambda_i|$. Given and arbitrary set $\{e_1, e_2\}$ of mutually orthogonal rank one elements in $S(C_1(H))$, the set $\{\sigma_1 e_1, \sigma_2 e_2\}$ satisfies the same properties. It follows from \eqref{eq 3 positive coef} that $$f(\lambda_1 e_1 + \lambda_2 e_2)= f(|\lambda_1| \sigma_1 e_1 + |\lambda_2| \sigma_2 e_2) = |\lambda_1| f(\sigma_1 e_1) + |\lambda_2| f(\sigma_2 e_2) $$ $$=\hbox{(by Proposition \ref{p minimal in arbitrary dimension}$(d)$ and $(e)$)} = |\lambda_1| \sigma_1 f( e_1) + |\lambda_2| \sigma_2 f( e_2) = \lambda_1 f( e_1) + \lambda_2 f( e_2).$$
\end{proof}

\begin{corollary}\label{c complex linearity and antilinearity} Let $f: S(C_1(H))\to S(C_1(H))$ be a surjective isometry, where $H$ is a two dimensional complex Hilbert space. Suppose $e_1, e_2$ is a (maximal) set of mutually orthogonal rank one elements in $S(C_1(H))$. The following statements hold:
\begin{enumerate}[$(a)$] \item If $f(i e_1) = i f(e_1)$ then $f(i e_2) = i f(e_2)$;
\item $f(i e_1) = - i f(e_1)$ then $f(i e_2) = - i f(e_2)$;
\end{enumerate}
\end{corollary}

\begin{proof}$(a)$ Suppose $f(i e_1) = i f(e_1)$. Let us find two orthonormal basis $\{\eta_1,\eta_2\},$ $\{\xi_1,\xi_2\}$ such that $e_1 = \eta_1\otimes \xi_1$ and $e_2= \eta_2\otimes \xi_2$. We set $u_1 = \eta_1\otimes \xi_2$ and $u_2= \eta_2\otimes \xi_1$. The elements $x = \frac12(e_1+e_2+u_1+u_2)$ and $y = \frac12(e_1+e_2-u_1-u_2)$ are rank one elements in $S(C_1(H))$.\smallskip

By Proposition \ref{p minimal in arbitrary dimension}$(d)$, $f(i x)= \pm i f(x)$ and $f(iy) =\pm i f(y)$.  If $f(i x)= - i f(x)$ we have $$\left\| f(i x) - f(i e_1)\right\|_1 = \left\| i x - i e_1\right\|_1 = \left\|  \left(
                                                                                              \begin{array}{cc}
                                                                                                -\frac12 & \frac12 \\
                                                                                                \frac12 & \frac12 \\
                                                                                              \end{array}
                                                                                            \right)
\right\|_1= \frac12 \left\|  \left(
                                                                                              \begin{array}{cc}
                                                                                                -1 & 1 \\
                                                                                                1 & 1 \\
                                                                                              \end{array}
                                                                                            \right)
\right\|_1= {\sqrt{2}}$$ and $$\left\| f(i x) - f(i e_1)\right\|_1 = \left\|-i f( x) - i f( e_1)\right\|_1= \left\| f(x) + f( e_1)\right\|_1 $$ $$= \left\| x + e_1\right\|_1= \left\|  \left(
                                                                                              \begin{array}{cc}
                                                                                                \frac32 & \frac12 \\
                                                                                                \frac12 & \frac12 \\
                                                                                              \end{array}
                                                                                            \right)
\right\|_1 = \frac12 \left\|  \left(
                                                                                              \begin{array}{cc}
                                                                                                3 & 1 \\
                                                                                                1 & 1 \\
                                                                                              \end{array}
                                                                                            \right)
\right\|_1= \frac12 \left(\sqrt{6+4\sqrt{2}}+ \sqrt{6-4\sqrt{2}}\right),$$ which gives a contradiction. Therefore, $f(i x)= i f(x)$. We similarly prove $f(iy) = i f(y)$.\smallskip

Since $i x\perp i y$ in $S(C_1(H))$, we can apply Proposition \ref{p affine real combinations dim 2} to deduce that $$f\left(\frac{i}{2} e_1+ \frac{i}{2} e_2 \right) = f\left(\frac{i}{2} x+ \frac{i}{2} y \right) = \frac12 f( i x) + \frac12 f(i y ) =  i \frac12 f(  x) + i \frac12 f( y )$$ $$ = i  f\left(\frac{1}{2} x+ \frac{1}{2} y \right)  = i f\left(\frac{1}{2} e_1+ \frac{1}{2} e_2 \right).$$ Since a new application of Proposition \ref{p affine real combinations dim 2} gives $f\left(\frac{1}{2} e_1+ \frac{1}{2} e_2 \right) = \frac{1}{2} f\left( e_1\right) + \frac{1}{2} f\left(e_2 \right)$ and $f\left(\frac{i}{2} e_1+ \frac{i}{2} e_2 \right) = \frac{1}{2} f\left(i e_1\right) + \frac{1}{2} f\left(i e_2 \right)$, the equality $f(i e_1) = i f(e_1)$ proves that $f(i e_2) = i f(e_2)$.\smallskip

We can actually prove that $f(i u_1) = i f(u_1)$ and $f(iu_2) = i f(u_2)$.\smallskip

Statement $(b)$ follows by similar arguments.
\end{proof}

Let $f: S(C_1(H))\to S(C_1(H))$ be a surjective isometry, where $H$ is a two dimensional complex Hilbert space. Suppose $e_1, e_2$ is a (maximal) set of mutually orthogonal rank one elements in $S(C_1(H))$ and $\lambda_1,\lambda_2$ are real numbers satisfying $|\lambda_1|+|\lambda_2|=1$. Then we have proved in Proposition \ref{p affine real combinations dim 2} that $$f(\lambda_1 e_1 + \lambda_2 e_2) = \lambda_1 f(e_1) + \lambda_2 f(e_2).$$

Arguing as in the proof of Proposition \ref{p affine real combinations dim 2} we can find two surjective complex linear isometries $T_1,T_2 : S(C_1(H))\to S(C_1(H))$ such that $g = T_2|_{S(C_1(H))} \circ f \circ T_1^{-1}|_{S(C_1(H))} : S(C_1(H)) \to S(C_1(H))$ is a surjective isometry satisfying $g(\phi_j) = \phi_j,$ where $\phi_1=\left(
                                                          \begin{array}{cc}
                                                            1 & 0 \\
                                                            0 & 0 \\
                                                          \end{array}
                                                        \right)$ and $\phi_2=\left(
                                                          \begin{array}{cc}
                                                            0 & 0 \\
                                                            0 & 1 \\
                                                          \end{array}
                                                        \right).$ Now, we claim that
\begin{equation}\label{eq counterdiagonal} g\left( \left(
                                                          \begin{array}{cc}
                                                            0 & 1 \\
                                                            0 & 0 \\
                                                          \end{array}
                                                        \right) \right) =  \left(
                                                          \begin{array}{cc}
                                                            0 & \mu \\
                                                            0 & 0 \\
                                                          \end{array}
                                                        \right)  \hbox{ and } g\left( \left(
                                                          \begin{array}{cc}
                                                            0 & 0 \\
                                                            1 & 0 \\
                                                          \end{array}
                                                        \right) \right) =  \left(
                                                          \begin{array}{cc}
                                                            0 & 0 \\
                                                            \overline{\mu} & 0 \\
                                                          \end{array}
                                                        \right)
\end{equation} or
\begin{equation}\label{eq counterdiagonal transpose} g\left( \left(
                                                          \begin{array}{cc}
                                                            0 & 1 \\
                                                            0 & 0 \\
                                                          \end{array}
                                                        \right) \right) =  \left(
                                                          \begin{array}{cc}
                                                            0 & 0 \\
                                                            \mu & 0 \\
                                                          \end{array}
                                                        \right)  \hbox{ and } g\left( \left(
                                                          \begin{array}{cc}
                                                            0 & 0 \\
                                                            1 & 0 \\
                                                          \end{array}
                                                        \right) \right) =  \left(
                                                          \begin{array}{cc}
                                                            0 & \overline{\mu} \\
                                                            0 & 0 \\
                                                          \end{array}
                                                        \right)
\end{equation} for suitable $\mu\in \mathbb{T}$. Indeed, we know from previous arguments that $g\left( \left(
                                                          \begin{array}{cc}
                                                            0 & 1 \\
                                                            0 & 0 \\
                                                          \end{array}
                                                        \right) \right)$ and $g\left( \left(
                                                          \begin{array}{cc}
                                                            0 & 0 \\
                                                            1 & 0 \\
                                                          \end{array}
                                                        \right) \right)$ are rank one orthogonal elements in $S(C_1(H))$. Let us write $g\left( \left(
                                                          \begin{array}{cc}
                                                            0 & 1 \\
                                                            0 & 0 \\
                                                          \end{array}
                                                        \right) \right) = \eta\otimes \xi = \left(
                                                          \begin{array}{cc}
                                                            \alpha & \beta \\
                                                            \gamma & \delta \\
                                                          \end{array}
                                                        \right)$, with $\alpha \delta = \gamma \beta$ and $|\alpha|^2 + |\beta|^2 + |\gamma|^2 + |\delta|^2 =1$.\smallskip

By hypothesis, $$\left\| \left(
                                                          \begin{array}{cc}
                                                            \alpha & \beta \\
                                                            \gamma & \delta \\
                                                          \end{array}
                                                        \right)\pm \left(
                                                          \begin{array}{cc}
                                                            1 & 0 \\
                                                            0 & 0 \\
                                                          \end{array}
                                                        \right) \right\|_1 =\left\| g\left( \left(
                                                          \begin{array}{cc}
                                                            0 & 1 \\
                                                            0 & 0 \\
                                                          \end{array}
                                                        \right) \right)\pm g\left( \left(
                                                          \begin{array}{cc}
                                                            1 & 0 \\
                                                            0 & 0 \\
                                                          \end{array}
                                                        \right) \right) \right\|_1 $$ $$= \left\| \left(
                                                          \begin{array}{cc}
                                                            0 & 1 \\
                                                            0 & 0 \\
                                                          \end{array}
                                                        \right)\pm  \left(
                                                          \begin{array}{cc}
                                                            1 & 0 \\
                                                            0 & 0 \\
                                                          \end{array}
                                                        \right) \right\|_1 = \sqrt{2}.$$

Since $\left\| \left(
                                                          \begin{array}{cc}
                                                            \alpha\pm 1 & \beta \\
                                                            \gamma & \delta \\
                                                          \end{array}
                                                        \right) \right\|_1^2 = 2$, we deduce from the inequality in \eqref{ineq Arazy} that $$2 \geq |\alpha \pm 1|^2 + |\beta|^2 +|\gamma|^2 +|\delta|^2 = 1+ |\alpha|^2 \pm 2 \Re\hbox{e} (\alpha) + |\beta|^2 +|\gamma|^2 +|\delta|^2,$$ which assures that $0\geq \pm 2 \Re\hbox{e} (\alpha)$ and hence $\Re\hbox{e} (\alpha) =0$. Replacing $g\left( \left(
                                                          \begin{array}{cc}
                                                            1 & 0 \\
                                                            0 & 0 \\
                                                          \end{array}
                                                        \right)\right)$ with $g\left( \left(
                                                          \begin{array}{cc}
                                                            i & 0 \\
                                                            0 & 0 \\
                                                          \end{array}
                                                        \right)\right)$ and having in mind that $g\left( i \left(
                                                          \begin{array}{cc}
                                                            1 & 0 \\
                                                            0 & 0 \\
                                                          \end{array}
                                                        \right)\right) \in \left\{ \pm i g\left( \left(
                                                          \begin{array}{cc}
                                                            1 & 0 \\
                                                            0 & 0 \\
                                                          \end{array}
                                                        \right)\right)\right\}$, we obtain $\Im\hbox{m} (\alpha)=0$, and thus $\alpha=0$.\smallskip

Similar arguments applied to $g\left(\left(
                                                          \begin{array}{cc}
                                                            0 & 0 \\
                                                            0 & 1 \\
                                                          \end{array}
                                                        \right)\right)$ instead of $g\left(\left(
                                                          \begin{array}{cc}
                                                            1 & 0 \\
                                                            0 & 0 \\
                                                          \end{array}
                                                        \right)\right)$ prove $\delta=0$.
Since $\gamma \beta=0$ and $|\beta|^2 + |\gamma|^2 =1$ we deduce that \begin{equation}\label{eq 11 pre}g\left( \left(
                                                          \begin{array}{cc}
                                                            0 & 1 \\
                                                            0 & 0 \\
                                                          \end{array}
                                                        \right) \right) =  \left(
                                                          \begin{array}{cc}
                                                            0 & \mu \\
                                                            0 & 0 \\
                                                          \end{array}
                                                        \right)  \hbox{ and } g\left( \left(
                                                          \begin{array}{cc}
                                                            0 & 0 \\
                                                            1 & 0 \\
                                                          \end{array}
                                                        \right) \right) =  \left(
                                                          \begin{array}{cc}
                                                            0 & 0 \\
                                                            \delta & 0 \\
                                                          \end{array}
                                                        \right)
\end{equation} or
\begin{equation}\label{eq 12 pre} g\left( \left(
                                                          \begin{array}{cc}
                                                            0 & 1 \\
                                                            0 & 0 \\
                                                          \end{array}
                                                        \right) \right) =  \left(
                                                          \begin{array}{cc}
                                                            0 & 0 \\
                                                            \mu & 0 \\
                                                          \end{array}
                                                        \right)  \hbox{ and } g\left( \left(
                                                          \begin{array}{cc}
                                                            0 & 0 \\
                                                            1 & 0 \\
                                                          \end{array}
                                                        \right) \right) =  \left(
                                                          \begin{array}{cc}
                                                            0 & \delta \\
                                                            0 & 0 \\
                                                          \end{array}
                                                        \right),
\end{equation} with $\delta,\mu\in \mathbb{T}$.\smallskip

We shall show next that $\delta = \overline{\mu}$. Indeed, let us assume that $g$ satisfies \eqref{eq 11 pre}. Then, applying Proposition \ref{p affine real combinations dim 2} we have $$\left\| \left(
                                                                                                                                      \begin{array}{cc}
                                                                                                                                        \frac12 & -\frac12 \\
                                                                                                                                        -\frac12 & \frac12 \\
                                                                                                                                      \end{array}
                                                                                                                                    \right) \right\|_1=  \left\| \frac12 \left(\left(
                                                                                                                                      \begin{array}{cc}
                                                                                                                                        1 & 0 \\
                                                                                                                                        0 & 0 \\
                                                                                                                                      \end{array}
                                                                                                                                    \right)+ \left(
    \begin{array}{cc}
      0 & 0 \\
      0 & 1 \\
    \end{array}
  \right)
 \right) - \frac12\left(\left(
             \begin{array}{cc}
               0 & 1 \\
               0 & 0 \\
             \end{array}
           \right) + \left(
                       \begin{array}{cc}
                         0 & 0 \\
                         1 & 0 \\
                       \end{array}
                     \right)
           \right)
 \right\|_1 $$
 $$= \left\| f\left(\frac12 \left(\left(
 \begin{array}{cc}
 1 & 0 \\
 0 & 0 \\
 \end{array}
 \right)
+ \left(
    \begin{array}{cc}
      0 & 0 \\
      0 & 1 \\
    \end{array}
  \right)
 \right) \right)- f\left( \frac12\left(\left(
             \begin{array}{cc}
               0 & 1 \\
               0 & 0 \\
             \end{array}
           \right) + \left(
                       \begin{array}{cc}
                         0 & 0 \\
                         1 & 0 \\
                       \end{array}
                     \right)
           \right)\right)
 \right\|_1 $$ $$= \left\| \left(
 \begin{array}{cc}
 \frac12 & -\frac{\mu}{2} \\
 -\frac{\delta}{2} & \frac12 \\
  \end{array}
  \right) \right\|_1 = \frac12 \left( \sqrt{2-|\overline{\delta} +\mu|} + \sqrt{2+|\overline{\delta} +\mu|}\right).$$ Since $$\left\| \left(
  \begin{array}{cc}
  \frac12 & -\frac12 \\
  -\frac12 & \frac12 \\
   \end{array}
  \right) \right\|_1=1,$$ we have $$ \sqrt{2-|\overline{\delta} +\mu|} + \sqrt{2+|\overline{\delta} +\mu|} = 2,$$ which implies $2=|\overline{\delta} +\mu|,$ and hence $\overline{\delta} =\mu$.\smallskip

Similarly, when we are in case \eqref{eq 12 pre} we get $\overline{\delta} =\mu$ and hence \eqref{eq counterdiagonal transpose} holds.\smallskip

Let us assume we are in the case derived from \eqref{eq counterdiagonal}, that is, $$g\left( \left(
                                                          \begin{array}{cc}
                                                            0 & 1 \\
                                                            0 & 0 \\
                                                          \end{array}
                                                        \right) \right) =  \left(
                                                          \begin{array}{cc}
                                                            0 & \mu \\
                                                            0 & 0 \\
                                                          \end{array}
                                                        \right)  \hbox{ and } g\left( \left(
                                                          \begin{array}{cc}
                                                            0 & 0 \\
                                                            1 & 0 \\
                                                          \end{array}
                                                        \right) \right) =  \left(
                                                          \begin{array}{cc}
                                                            0 & 0 \\
                                                            \overline{\mu} & 0 \\
                                                          \end{array}
                                                        \right).$$

Consider the unitary $v=\left(
         \begin{array}{cc}
           \sqrt{\mu} & 0 \\
           0 & \overline{\sqrt{\mu}} \\
         \end{array}
       \right)$ and the surjective linear isometry $T: S_1(H)\to S_1(H)$, $T(x) = v^* x v$. It is easy to see that $T\left(\left(
                                \begin{array}{cc}
                                  1 & 0 \\
                                  0 & 0 \\
                                \end{array}
                              \right)
        \right) =  \left(
                                \begin{array}{cc}
                                  1 & 0 \\
                                  0 & 0 \\
                                \end{array}
                              \right),$ $T\left(\left(
                                \begin{array}{cc}
                                  0 & 0 \\
                                  0 & 1 \\
                                \end{array}
                              \right)
        \right) =  \left(
                                \begin{array}{cc}
                                  0 & 0 \\
                                  0 & 1 \\
                                \end{array}
                              \right),$ $T\left(\left(
                                \begin{array}{cc}
                                  0 & \mu \\
                                  0 & 0 \\
                                \end{array}
                              \right)
        \right) =  \left(
                                \begin{array}{cc}
                                  0 & 1 \\
                                  0 & 0 \\
                                \end{array}
                              \right)$ and $T\left(\left(
                                \begin{array}{cc}
                                  0 & 0 \\
                                  \overline{\mu} & 0 \\
                                \end{array}
                              \right)
        \right) =  \left(
                                \begin{array}{cc}
                                  0 & 0 \\
                                  1 & 0 \\
                                \end{array}
                              \right).$
Therefore, replacing $g$ with $h = T|_{S(S_1(H))} g$, we obtain a surjective isometry $h : S(S_1(H))\to S(S_1(H))$ satisfying $h\left(\left(
                                \begin{array}{cc}
                                  1 & 0 \\
                                  0 & 0 \\
                                \end{array}
                              \right)
        \right) =  \left(
                                \begin{array}{cc}
                                  1 & 0 \\
                                  0 & 0 \\
                                \end{array}
                              \right),$ \linebreak $h\left(\left(
                                \begin{array}{cc}
                                  0 & 0 \\
                                  0 & 1 \\
                                \end{array}
                              \right)
        \right) =  \left(
                                \begin{array}{cc}
                                  0 & 0 \\
                                  0 & 1 \\
                                \end{array}
                              \right),$ $h\left(\left(
                                \begin{array}{cc}
                                  0 & 1 \\
                                  0 & 0 \\
                                \end{array}
                              \right)
        \right) =  \left(
                                \begin{array}{cc}
                                  0 & 1 \\
                                  0 & 0 \\
                                \end{array}
                              \right)$ and $h\left(\left(
                                \begin{array}{cc}
                                  0 & 0 \\
                                  1 & 0 \\
                                \end{array}
                              \right)
        \right) =  \left(
                                \begin{array}{cc}
                                  0 & 0 \\
                                  1 & 0 \\
                                \end{array}
                              \right).$\smallskip

When \eqref{eq counterdiagonal transpose} holds we can find surjective linear isometry $T: S_1(H)\to S_1(H)$ such that $h = T|_{S(S_1(H))} g$ is a surjective isometry satisfying $h\left(\left(
                                \begin{array}{cc}
                                  1 & 0 \\
                                  0 & 0 \\
                                \end{array}
                              \right)
        \right) =  \left(
                                \begin{array}{cc}
                                  1 & 0 \\
                                  0 & 0 \\
                                \end{array}
                              \right),$ $h\left(\left(
                                \begin{array}{cc}
                                  0 & 0 \\
                                  0 & 1 \\
                                \end{array}
                              \right)
        \right) =  \left(
                                \begin{array}{cc}
                                  0 & 0 \\
                                  0 & 1 \\
                                \end{array}
                              \right),$ $h\left(\left(
                                \begin{array}{cc}
                                  0 & 1 \\
                                  0 & 0 \\
                                \end{array}
                              \right)
        \right) =  \left(
                                \begin{array}{cc}
                                  0 & 0 \\
                                  1 & 0 \\
                                \end{array}
                              \right)$ and \linebreak $h\left(\left(
                                \begin{array}{cc}
                                  0 & 0 \\
                                  1 & 0 \\
                                \end{array}
                              \right)
        \right) =  \left(
                                \begin{array}{cc}
                                  0 & 1 \\
                                  0 & 0 \\
                                \end{array}
                              \right).$\smallskip

We shall establish now a technical proposition to measure the distance between two pure atoms in $C_1(H)$. For each pure atom $e=\eta\otimes \xi$ in $S(C_1(H))$, as before, let $s(e)=\eta\otimes \xi$ be the unique minimal partial isometry in $B(H)$ satisfying $e (s(e)) = 1$. For any other $x\in S(C_1(H))$ the evaluation $x(s(e))\in \mathbb{C}$. For each partial isometry $s$ in $B(H)$, the Bergmann operator $P_0(s) : B(H)\to B(H),$ $x\mapsto (1-ss^*) x(1-s^*s)$, is weak$^*$ continuous and hence $P_0(s)^* (x)\in C_1(H)$ for every $x\in C_1(H)$.

\begin{lemma}\label{l distance between pure atoms} Let $H$ be a complex Hilbert space. Suppose $e_1, e_2$ are two rank one pure atoms in $S(C_1(H))$. Then the following formula holds: $$\| e_2-e_1 \|_{1} = \sum_{k=1,2} \sqrt{(1-  \Re\hbox{e} (\alpha)) +(-1)^k \sqrt{ (1-  \Re\hbox{e} (\alpha))^2 - \hat\delta^2}} ,$$ where $\alpha=\alpha (e_1,e_2) = e_2 (s(e_1))$ and $\hat\delta =\hat\delta (e_1,e_2) = \|P_0(s(e_1))^* (e_2)\|_1$.
\end{lemma}

\begin{proof} By choosing an appropriate matrix representation we can find two orthonormal systems $\{\eta_1, \eta_2\}$ and $\{\xi_1, \xi_2\}$ to
represent $e_1$ and $e_2$ in the form $e_1 = \eta_1\otimes \xi_1$, $e_2 = \widetilde{\eta}_1\otimes \widetilde{\xi}_1$ and $$e_2=\alpha  v_{11}+ \beta v_{12}+ \delta v_{22}+\gamma v_{21},$$ where
$e_1= v_{11}$, $v_{12} = {\eta}_2 \otimes {\xi}_1$, $v_{21} = {\eta}_1 \otimes {\xi}_2$, $v_{22} = {\eta}_2 \otimes {\xi}_2$,
$\alpha= \langle \xi_1 / \widetilde{\xi}_1 \rangle \langle  \widetilde{\eta}_1 / \eta_1 \rangle,$ $ \beta = \langle \xi_1 / \widetilde{\xi}_1 \rangle \langle  \widetilde{\eta}_1 / \eta_2 \rangle,$  $\gamma= \langle \xi_2 / \widetilde{\xi}_1 \rangle \langle  \widetilde{\eta}_1 / \eta_1 \rangle,$ $\delta = \langle \xi_2 / \widetilde{\xi}_1 \rangle \langle  \widetilde{\eta}_1 / \eta_2 \rangle \in \mathbb{C},$ with $|\alpha|^2+ |\beta|^2+|\gamma|^2+ |\delta|^2$ $=|\langle \xi_1 / \widetilde{\xi}_1 \rangle|^2 \|\widetilde{\eta}_1\|^2+ |\langle \xi_2 / \widetilde{\xi}_1 \rangle|^2 \|\widetilde{\eta}_1 \|^2 = \| \widetilde{\xi}_1\|^2=1$, and $\alpha \delta= \beta \gamma$. In an appropriate matrix representation we can identify $e_1$ and $e_2$ with $\left(
                  \begin{array}{cc}
                    1 & 0 \\
                    0 & 0 \\
                  \end{array}
                \right),$ and $\left(
                  \begin{array}{cc}
                    \alpha & \beta \\
                    \gamma & \delta \\
                  \end{array}
                \right),$ respectively.\smallskip

Following the arguments in the proof of \cite[Proposition 3.3]{FerPe17} we deduce that the eigenvalues of the element $(e_2-e_1) (e_2-e_1)^* = \left(
                  \begin{array}{cc}
                    \alpha-1 & \beta \\
                    \gamma & \delta \\
                  \end{array}
                \right) \left(
                  \begin{array}{cc}
                    \alpha-1 & \beta \\
                    \gamma & \delta \\
                  \end{array}
                \right)^*$ are precisely  $(1-  \Re\hbox{e} (\alpha)) \pm\sqrt{ (1-  \Re\hbox{e} (\alpha))^2 - |\delta|^2}$ and hence $$\| e_2-e_1 \|_1 = \sum_{k=1,2} \sqrt{(1-  \Re\hbox{e} (\alpha)) +(-1)^k \sqrt{ (1-  \Re\hbox{e} (\alpha))^2 - |\delta|^2}} ,$$ which proves the desired formula. \end{proof}

We are now in position to solve Tingley's problem for the case of trace class operators on a two dimensional Hilbert space.

\begin{theorem}\label{t Tingley trace class dim 4} Let $f: S(C_1(H))\to S(C_1(H))$ be a surjective isometry, where $H$ is a two dimensional complex Hilbert space. Then there exists a surjective complex linear or conjugate linear isometry $T :C_1(H)\to C_1(H)$  satisfying $f(x) = T(x)$ for every $x\in S(C_1(H))$. More concretely, there exist unitaries $u,v\in M_2(\mathbb{C})$ such that one of the following statements holds:\begin{enumerate}[$(a)$] \item $f(x) = u x v$, for every $x\in S(C_1(H))$;
\item $f(x) = u x^t v$, for every $x\in S(C_1(H))$;
\item $f(x) = u \overline{x} v$, for every $x\in S(C_1(H))$;
\item $f(x) = u x^* v$, for every $x\in S(C_1(H))$,
\end{enumerate} where $\overline{(x_{ij})} = (\overline{x_{ij}})$.
\end{theorem}

\begin{proof} By the comments preceding this theorem, we can find two surjective linear isometries $U,V :C_1(H)\to C_1(H)$ such that the mapping $h=U|_{S(C_1(H))} \circ f\circ V|_{S(C_1(H))} : S(C_1(H))\to S(C_1(H))$ is a surjective isometry satisfying precisely one of the next statements \begin{equation}\label{eq case A 11} h\left(\left(
                                \begin{array}{cc}
                                  1 & 0 \\
                                  0 & 0 \\
                                \end{array}
                              \right)
        \right) =  \left(
                                \begin{array}{cc}
                                  1 & 0 \\
                                  0 & 0 \\
                                \end{array}
                              \right), \ h\left(\left(
                                \begin{array}{cc}
                                  0 & 1 \\
                                  0 & 0 \\
                                \end{array}
                              \right)
        \right) =  \left(
                                \begin{array}{cc}
                                  0 & 1 \\
                                  0 & 0 \\
                                \end{array}
                              \right),
\end{equation} $$h\left(\left(
                                \begin{array}{cc}
                                  0 & 0 \\
                                  1 & 0 \\
                                \end{array}
                              \right)
        \right) =  \left(
                                \begin{array}{cc}
                                  0 & 0 \\
                                  1 & 0 \\
                                \end{array}
                              \right), \hbox{ and } h\left(\left(
                                \begin{array}{cc}
                                  0 & 0 \\
                                  0 & 1 \\
                                \end{array}
                              \right)
        \right) =  \left(
                                \begin{array}{cc}
                                  0 & 0 \\
                                  0 & 1 \\
                                \end{array}
                              \right);$$ or
\begin{equation}\label{eq case B 12} h\left(\left(
                                \begin{array}{cc}
                                  1 & 0 \\
                                  0 & 0 \\
                                \end{array}
                              \right)
        \right) =  \left(
                                \begin{array}{cc}
                                  1 & 0 \\
                                  0 & 0 \\
                                \end{array}
                              \right), \ h\left(\left(
                                \begin{array}{cc}
                                  0 & 1 \\
                                  0 & 0 \\
                                \end{array}
                              \right)
        \right) =  \left(
                                \begin{array}{cc}
                                  0 & 0 \\
                                  1 & 0 \\
                                \end{array}
                              \right),
\end{equation} $$h\left(\left(
                                \begin{array}{cc}
                                  0 & 0 \\
                                  1 & 0 \\
                                \end{array}
                              \right)
        \right) =  \left(
                                \begin{array}{cc}
                                  0 & 1 \\
                                  0 & 0 \\
                                \end{array}
                              \right), \hbox{ and } h\left(\left(
                                \begin{array}{cc}
                                  0 & 0 \\
                                  0 & 1 \\
                                \end{array}
                              \right)
        \right) =  \left(
                                \begin{array}{cc}
                                  0 & 0 \\
                                  0 & 1 \\
                                \end{array}
                              \right).$$

Proposition \ref{p minimal in arbitrary dimension}$(d)$ and $(e)$ implies that $f(- z) = -f(z)$ and $h(- z) = -h(z)$ for every pure atom $z$ in $C_1(H)$.\smallskip

We assume that \eqref{eq case A 11} holds. Let us denote
$v_{11} = \left(
            \begin{array}{cc}
              1 & 0 \\
              0 & 0 \\
            \end{array}
          \right)$, $v_{12} = \left(
            \begin{array}{cc}
              0 & 1 \\
              0 & 0 \\
            \end{array}
          \right)$,  $v_{21} = \left(
            \begin{array}{cc}
              0 & 0 \\
              1 & 0 \\
            \end{array}
          \right)$, and $v_{22} = \left(
            \begin{array}{cc}
              0 & 0 \\
              0 & 1 \\
            \end{array}
          \right)$.\smallskip

By Proposition \ref{p minimal in arbitrary dimension}$(d)$ and Corollary \ref{c complex linearity and antilinearity}, we know that one of the next statements holds: \begin{enumerate}[$(a)$] \item $h(i v_{11}) = i h (v_{11})$ and $h(i v_{22}) = i h (v_{22})$;
\item $h(i v_{11}) = -i h (v_{11})$ and $h(i v_{22}) = -i h (v_{22})$.
\end{enumerate} The proof will be splitted into two cases corresponding to the above statements.\smallskip

\emph{Case $(a)$.} Let us assume that $(a)$ holds. We consider the pure atom
$e_1=\left(                                                                                    \begin{array}{cc}
\frac12 & \frac12 \\                                                                                      \frac12 & \frac12 \\                                                                \end{array} \right)$. Proposition \ref{p minimal in arbitrary dimension} assures that $h(e_1)$ is a pure atom, and hence it must be of the form $h(v) = \left(
                  \begin{array}{cc}
                    \alpha' & \beta' \\
                    \gamma' & \delta' \\
                  \end{array}
                \right)$ with $|\alpha'|^2+ |\beta'|^2+|\gamma'|^2+ |\delta'|^2=1$, $\alpha' \delta'= \beta' \gamma'$. By the hypothesis on $h$ and Lemma \ref{l distance between pure atoms} we get
$$ \sqrt{2} =\sum_{k=1,2} \sqrt{(1-  \frac12 ) +(-1)^k \sqrt{ (1-  \frac12 )^2 - \left(\frac12\right)^2}} = \| e_1-v_{11} \|_1$$
$$ = \| h(e_1)-v_{11} \|_1 = \sum_{k=1,2} \sqrt{(1-  \Re\hbox{e} (\alpha')) +(-1)^k \sqrt{ (1-  \Re\hbox{e} (\alpha'))^2 - |\delta'|^2}},$$ and $$\sqrt{\frac32 +\sqrt{2}} + \sqrt{\frac32 -\sqrt{2}}= \sum_{k=1,2} \sqrt{(1+  \frac12) +(-1)^k \sqrt{ (1+  \frac12)^2 - \left(\frac12\right)^2}} = \| e_1 +v_{11} \|_1$$
$$ = \| h(e_1)-h(-v_{11}) \|_1 = \sum_{k=1,2} \sqrt{(1+  \Re\hbox{e} (\alpha')) +(-1)^k \sqrt{ (1+  \Re\hbox{e} (\alpha'))^2 - |\delta'|^2}}.$$ Taking squares in both sides we get $$ 2 = 2 (1-  \Re\hbox{e} (\alpha')) + 2 |\delta'|,$$ and $$  4= 2 (1+  \Re\hbox{e} (\alpha')) + 2 |\delta'|,$$ which gives $\Re\hbox{e} (\alpha') = \frac12$ and $ |\delta'|=\frac12$.\smallskip

When in the above arguments we replace $v_{11}$ with $v_{12}$, $v_{21}$ and $v_{22}$ we obtain $\Re\hbox{e} (\beta') = \frac12$, $\Re\hbox{e} (\gamma') = \frac12$ and $\Re\hbox{e} (\delta') = \frac12$. Since  $|\alpha'|^2+ |\beta'|^2+|\gamma'|^2+ |\delta'|^2=1$, we deduce that $\alpha' = \beta'=\gamma'= \delta' = \frac12$, and hence $h(e_1) = e_1 = \left(                                                                                    \begin{array}{cc}
\frac12 & \frac12 \\                                                                                      \frac12 & \frac12 \\                                                                \end{array} \right).$\smallskip

We can similarly show that taking $e_2 = \left(                                                                                    \begin{array}{cc}
\frac12 & -\frac12 \\                                                                                      -\frac12 & \frac12 \\                                                                \end{array} \right)$, we have $h(e_2) = e_2$.\smallskip

Now since $e_1$ and $e_2$ are orthogonal pure atoms, Proposition \ref{p minimal in arbitrary dimension}$(d)$ and $(e)$ and Corollary \ref{c complex linearity and antilinearity} imply that exactly one of the next statements holds: \begin{enumerate} \item[$(a.1)$] $h(i e_1) = i h (e_{1})$ and $h(i e_{2}) = i h (e_{2})$;
\item[$(a.2)$] $h(i e_{1}) = -i h (e_{1})$ and $h(i e_{2}) = -i h (e_{2})$.
\end{enumerate}

Let us show that the conclusion in $(a.2)$ is impossible. Indeed, if $(a.2)$ holds, Proposition \ref{p affine real combinations dim 2} implies $$\frac12 i v_{11} + \frac12 i v_{22} = \frac12  h\left(i v_{11} \right) + \frac12 h\left( i v_{22} \right)  = h\left(\frac12 i v_{11} + \frac12 i v_{22} \right) $$ $$ =h\left(\frac12 i e_1 + \frac12 i e_2 \right) = \frac12 h( i e_1 ) + \frac12 h(i e_2 ) = - \frac12 i e_1 - \frac12 i e_2 = -\frac12 i v_{11} - \frac12 i v_{22} ,$$ which is impossible.\smallskip

Since $(a.1)$ holds, we deduce, via  Proposition \ref{p affine real combinations dim 2}, that $$\frac12 h\left( i v_{12} \right) + \frac12 h\left(i v_{21} \right)= h\left(\frac12 i v_{12} + \frac12 i v_{21} \right)= h\left(\frac12 i e_1 - \frac12 i e_2 \right) $$ $$ = \frac12 h\left( i e_1 \right)- \frac12 h\left( i e_2 \right) = \frac12 i e_1 - \frac12 i e_2 = \frac12 i v_{12} + \frac12 i v_{21}.$$ We know from Corollary \ref{c complex linearity and antilinearity} that $h\left( i v_{jk} \right) \in \{\pm i h(v_{jk})\} = \{\pm i v_{jk}\}$, for every $k,j=1,2$. Thus, \begin{equation}\label{eq complex lin counterdiagonal} h\left(i v_{12} \right) = i  h\left( v_{12} \right) = i v_{12},\hbox{ and }  h\left(i v_{21} \right) = i  h\left( v_{21} \right) = i v_{21}.
\end{equation}

We shall prove that

\begin{equation}\label{eq h coincides with Id} h(v) = v, \hbox{ for every pure atom $v\in S(C_1(H))$.}
\end{equation}

Let $v$ be a pure atom (i.e. a rank one partial isometry) in $S(C_1(H))$. By Proposition \ref{p minimal in arbitrary dimension}$(b)$, $h(v)$ is a pure atom in $S(C_1(H))$. Arguing as in the proof of Lemma \ref{l distance between pure atoms}, we may assume that $v= \left(
                  \begin{array}{cc}
                    \alpha & \beta \\
                    \gamma & \delta \\
                  \end{array}
                \right),$ and $h(v)= \left(
                  \begin{array}{cc}
                    \alpha' & \beta' \\
                    \gamma' & \delta' \\
                  \end{array}
                \right),$ with $|\alpha|^2+ |\beta|^2+|\gamma|^2+ |\delta|^2=1$, $|\alpha'|^2+ |\beta'|^2+|\gamma'|^2+ |\delta'|^2=1$, $\alpha' \delta'= \beta' \gamma'$, and $\alpha \delta= \beta \gamma$.\smallskip

Applying the hypothesis on $h$ and Lemma \ref{l distance between pure atoms} we get the following equations
$$ \sum_{k=1,2} \sqrt{(1-  \Re\hbox{e} (\alpha)) +(-1)^k \sqrt{ (1-  \Re\hbox{e} (\alpha))^2 - |\delta|^2}} = \| v-v_{11} \|_1$$
$$ = \| h(v)-v_{11} \|_1 = \sum_{k=1,2} \sqrt{(1-  \Re\hbox{e} (\alpha')) +(-1)^k \sqrt{ (1-  \Re\hbox{e} (\alpha'))^2 - |\delta'|^2}},$$ and $$ \sum_{k=1,2} \sqrt{(1+  \Re\hbox{e} (\alpha)) +(-1)^k \sqrt{ (1+  \Re\hbox{e} (\alpha))^2 - |\delta|^2}} = \| v +v_{11} \|_1$$
$$ = \| h(v)-h(-v_{11}) \|_1 = \sum_{k=1,2} \sqrt{(1+  \Re\hbox{e} (\alpha')) +(-1)^k \sqrt{ (1+  \Re\hbox{e} (\alpha'))^2 - |\delta'|^2}}.$$ Taking squares in both sides we get $$ 2 (1-  \Re\hbox{e} (\alpha)) + 2 |\delta|= 2 (1-  \Re\hbox{e} (\alpha')) + 2 |\delta'|,$$ and $$ 2 (1+  \Re\hbox{e} (\alpha)) + 2 |\delta|= 2 (1+  \Re\hbox{e} (\alpha')) + 2 |\delta'|,$$ which gives $\Re\hbox{e} (\alpha') = \Re\hbox{e} (\alpha)$ and $|\delta | = |\delta'|$.\smallskip

Then applying the above arguments to $v$, $i v_{11}$ and $-i v_{11}$ we obtain $\Im\hbox{m} (\alpha') = \Im\hbox{m} (\alpha)$, and hence $\alpha = \alpha'$.\smallskip

Having in mind that in case $(a)$ we have $h(i v_{jk}) = i v_{jk}$ for every $j,k$ (see \eqref{eq complex lin counterdiagonal}), a similar reasoning applied to $v$, $v_{22}$ and $i v_{22}$ (respectively, $v$, $v_{12}$ and $i v_{12}$  or $v$, $v_{21}$ and $i v_{21}$) gives $\delta = \delta'$ (respectively, $\beta = \beta'$ or $\gamma = \gamma'$). We have therefore shown that $h(v) = v$, for every pure atom $v$.\smallskip

Proposition \ref{p affine real combinations dim 2} assures that $h(x) = x$ for every $x\in S(C_1(H))$, and hence $f = U^{-1} V^{-1}|_{S(C_1(H))},$ where $U^{-1} V^{-1}= C_1(H)\to C_1(H)$ is a surjective complex linear isometry.\smallskip

In \emph{Case $(b)$,} we can mimic the above arguments to show that \begin{equation}\label{eq h coincides with Idc} h(v) = \overline{v}, \hbox{ for every pure atom $v\in S(C_1(H))$,}
\end{equation} where $\overline{\left(
\begin{array}{cc}
x_{11} & x_{12} \\
x_{21} & x_{21} \\
\end{array}
\right)}= \left(
\begin{array}{cc}
\overline{x_{11}} & \overline{x_{12}} \\
\overline{x_{21}} & \overline{x_{21}} \\
\end{array}
\right)
$, and consequently, $f(x) = U^{-1} (\overline{V^{-1} (x)})=T(x)$, for every $x\in S(C_1(H))$, where $T : C_1(H)\to C_1(H)$, $T(x) =U^{-1}( \overline{V^{-1} (x)})$ ($x\in C_1(H)$) is a surjective conjugate linear isometry.\smallskip

Finally, if we assume \eqref{eq case B 12}, then there exist two surjective linear isometries $U,V :C_1(H)\to C_1(H)$ such that $$f(x) = U^{-1}( {V^{-1} (x^t)})$$ or $$f(x) = U^{-1}( {V^{-1} (x^*)})$$ for every $x\in S(C_1(H)).$
\end{proof}

Before dealing with surjective isometries between the unit spheres of trace class spaces over a finite dimensional complex Hilbert space, we shall present a technical result.

\begin{proposition}\label{p surjective isometry identity on non maximal elements} Let $f: S(C_1(H))\to S(C_1(H))$ be a surjective isometry, where $H$ is a complex Hilbert space with dim$(H)=n$. Suppose $f$ satisfies the following property: given a set $\{ e_1, e_2, \ldots, e_k\}$ of mutually orthogonal pure atoms in $S(C_1(H))$ with $k< n$, and real numbers $\lambda_1,\lambda_2, \ldots, \lambda_k$ satisfying $\displaystyle \sum_{j=1}^{k} |\lambda_j|=1$ we have $\displaystyle  f\left( \sum_{j=1}^{k} \lambda_j e_j \right) = \sum_{j=1}^{k} \lambda_j e_j$. Then $f(x)=x,$ for every $x\in S(C_1(H))$.
\end{proposition}

\begin{proof} Let $\{\xi_1,\ldots,\xi_n\}$ be an orthonormal basis of $H$. We set $v_j = \xi_j\otimes \xi_j$ ($j\in \{1,\ldots,n\}$). We claim that the identity \begin{equation}\label{eq identity on main diagonal tech lemma} f\left( \sum_{j=1}^{n} \mu_j v_j \right) = \sum_{j=1}^{n} \mu_j v_j,
\end{equation} holds for every $\mu_1, \ldots, \mu_n$ in $\mathbb{R}^{+}_0$ with $\displaystyle \sum_{j=1}^{n} \mu_j=1$. We observe that we can assume that $\mu_j>0$ for every $j$, otherwise the statement is clear from the hypothesis on $f$.\smallskip

By the hypothesis on $f$ we know that $f(v_j) =v_j$ for every $j\in\{1,\ldots, n\}$. Let $1$ denote the identity in $B(H)$. Clearly $\{ v_1, v_2, \ldots, v_n\}\subseteq \{1\}_{\prime}.$ By Proposition \ref{p minimal in arbitrary dimension} (see also \eqref{eq general form of a face finite dimension}) there exists a unique unitary $w\in B(H)$ such that $f(\{1\}_{\prime}) =\{w\}_{\prime}$. Since, for each $j=1,\ldots, n$, $v_j= f(v_j)\in \{w\}_{\prime},$ we can easily deduce that $w =1$ and hence $f(\{1\}_{\prime}) =\{1\}_{\prime}$ is the face of all states (i.e., positive norm-one functional on $B(H)$).\smallskip

The element  $\displaystyle f\left( \sum_{j=1}^{n} \mu_j v_j \right)\in  f(\{1\}_{\prime}) =\{1\}_{\prime}$, and hence there exists a positive matrix $a= (a_{ij})\in M_n(\mathbb{C})$ with $\|a\|_1 = 1= \hbox{tr} (a)$ such that $\displaystyle f\left( \sum_{j=1}^{n} \mu_j v_j \right) = a.$ It should be remarked that we can also identify each $v_j$ with the matrix in $M_n(\mathbb{C})$ with entry $1$ in the $(j,j)$ position and zero otherwise.\smallskip

Let us fix a projection $p\in B(H)$. The mapping $M_p: C_1(H)\to C_1(H)$, $M_p (x) = p x p + (1-p) x (1-p)$ is linear, contractive and positive. Let $p_j$ denote the projection $\xi_j\otimes \xi_j = s(v_j) \in B(H)$. By hypothesis $$\displaystyle \|a - v_{j}\|_1 = \left\| f\left( \sum_{j=1}^{n} \mu_j v_j \right) -f (v_j)  \right\|_1 = \left\|  \sum_{j=1}^{n} \mu_j v_j   - v_j  \right\|_1 $$ $$= \sum_{k\neq j} \mu_k + 1-\mu_j = 2 \sum_{k\neq j} \mu_k = 2(1-\mu_j).$$ Having in mind that $ (1-{p_j}) a (1-{p_j})$, $ {p_j} a {p_j},$ and $a$ are positive functionals with $a(1) = \|a\|_1 = 1$, we get $$ 2(1-\mu_j) = \|a - v_{j}\|_1\geq \|M_{p_j} (a - v_{j}) \|_1 = \|{p_j} (a - v_{j}) {p_j} + (1-{p_j}) a (1-{p_j}) \|_1$$ $$=\hbox{(by orthogonality)}= \|{p_j} (a - v_{j}) {p_j} \|_1 + \| (1-{p_j}) a (1-{p_j}) \|_1 $$ $$= 1-a_{jj} +  ((1-{p_j}) a (1-{p_j})) (1-p_j) = 1-a_{jj} + a (1-p_j) = 2 (1-a_{jj}).$$ This shows that $a_{jj} \geq \mu_j$ for every $j=1,\ldots, n$. Since $$1 = a(1) = a_{11}+\ldots+a_{nn} \geq \mu_1+\ldots+\mu_n =1,$$ we deduce that $a_{jj} = \mu_j$ for every $j=1,\ldots, n$.\smallskip

We shall now show that $a_{ij} = 0$ for every $i\neq j$. For this purpose, fix $i\neq j$ and set $q = p_i +p_j\in B(H)$ and $x = \mu_i v_i + (1-\mu_i) v_j\in S(C_1(H)).$ By hypothesis $f(x) =x$, and $$ \|a - x \|_1 = \left\|  \sum_{j=1}^{n} \mu_j v_j   - x \right\|_1= \sum_{k\neq i,j} \mu_k + 1-\mu_i-\mu_j = 2 (1-\mu_i-\mu_j ).$$ Let us observe that $$\left(\left(
                                                                        \begin{array}{cc}
                                                                         \mu_i & a_{ij} \\
                                                                        \overline{a_{ij}} & \mu_j
                                                                        \end{array}
                                                                      \right) - \left(
                                                                                  \begin{array}{cc}
                                                                                    \mu_i & 0 \\
                                                                                    0 & 1-\mu_i \\
                                                                                  \end{array}
                                                                                \right)\right)^2 = \left(
                                                                        \begin{array}{cc}
                                                                         0 & a_{ij} \\
                                                                        \overline{a_{ij}} & - 1+\mu_i +\mu_j
                                                                        \end{array}
                                                                      \right)^2$$ whose eigenvalues are precisely
$$\sqrt{\frac{(1-\mu_i-\mu_j)^2 + 2 |a_{ij}|^2 \pm \sqrt{((1-\mu_i-\mu_j)^2 + 2 |a_{ij}|^2)^2 - 4 |a_{ij}|^4}}{2}},$$ and thus $$\left\| \left(
                                                                        \begin{array}{cc}
                                                                         \mu_i & a_{ij} \\
                                                                        \overline{a_{ij}} & \mu_j
                                                                        \end{array}
                                                                      \right) - \left(
                                                                                  \begin{array}{cc}
                                                                                    \mu_i & 0 \\
                                                                                    0 & 1-\mu_i \\
                                                                                  \end{array}
                                                                                \right)
 \right\|_1^2 = (1-\mu_i-\mu_j)^2 + 2 |a_{ij}|^2 + 2 |a_{ij}|^2 $$ $$=  (1-\mu_i-\mu_j)^2 + 4 |a_{ij}|^2.$$

Therefore, we have $$2 (1-\mu_i-\mu_j ) \geq  \|M_{q} (a - x) \|_1 = \|q (a - x) q + (1-q) a (1-q) \|_1$$ $$=\hbox{(by orthogonality)}= \|q (a - x) q \|_1 + \| (1-q) a (1-q) \|_1 $$ $$= \left\| \left(
                                                                        \begin{array}{cc}
                                                                         \mu_i & a_{ij} \\
                                                                        \overline{a_{ij}} & \mu_j
                                                                        \end{array}
                                                                      \right) - \left(
                                                                                  \begin{array}{cc}
                                                                                    \mu_i & 0 \\
                                                                                    0 & 1-\mu_i \\
                                                                                  \end{array}
                                                                                \right)
 \right\|_1 + a (1-q) $$ $$= \sqrt{(1-\mu_i-\mu_j)^2 + 4 |a_{ij}|^2} + a(1) - a(p_i) - a(p_j)  $$ $$= \sqrt{(1-\mu_i-\mu_j)^2 + 4 |a_{ij}|^2} + 1-\mu_i-\mu_j,$$ which implies that $1-\mu_i-\mu_j \geq \sqrt{(1-\mu_i-\mu_j)^2 + 4 |a_{ij}|^2},$ and hence $a_{ij} =0$ as desired. We have thus proved that $$ f\left( \sum_{j=1}^{n} \mu_j v_j \right) = a =  \sum_{j=1}^{n} \mu_j v_j,$$ which concludes the proof of \eqref{eq identity on main diagonal tech lemma}.\smallskip

Finally, let us take $x\in S(C_1(H))$. We can find a set $\{ e_1, e_2, \ldots, e_n\}$ of mutually orthogonal pure atoms in $S(C_1(H))$ and real numbers $\lambda_1,\lambda_2, \ldots, \lambda_n$ such that $\displaystyle x = \sum_{j=1}^{n} \lambda_j e_j.$ We observe that, replacing each $e_j$ with $\pm e_j$ we can always assume that $\lambda_j\geq 0$ for every $j$. Let us pick two unitary matrices $u_1, w_1\in B(H)$ satisfying $ u_1 v_j w_1 = e_j$ for every $j=1,\ldots, n$. The operator $T_{u_1,w_1}: C_1(H)\to C_1(H)$, $T(y) = u_1 y w_1$ is a surjective isometry mapping elements of rank $k$ to elements of the same rank. Consequently, the mapping $f_2 : S(C_1(H))\to S(C_1(H)),$ $f_2 (y) = u_1^* f(u_1 y w_1) w_1^*$ is a surjective isometry. For each $y\in S(C_1(H))$ with rank$(y)<n$, we have $u_1 y w_1\in S(C_1(H))$ with rank$(u_1 y w_1)<n$ and thus, by the hypothesis on $f$, $f(u_1 y w_1) = u_1 y w_1$, which implies $f_2 (y) = y$. Therefore $f_2$ satisfies the same hypothesis of $f$. Applying \eqref{eq identity on main diagonal tech lemma} we get $$u_1^* f(x) w_1^*= u_1^* f \left( \sum_{j=1}^{n} \lambda_j e_j \right)  w_1^*= u_1^* f \left( u_1 \left(\sum_{j=1}^{n} \lambda_j v_j\right) w_1 \right) w_1^* $$ $$= f_2 \left( \sum_{j=1}^{n} \lambda_j v_j \right) = \sum_{j=1}^{n} \lambda_j v_j, $$ which proves that $\displaystyle f(x)= f \left( \sum_{j=1}^{n} \lambda_j e_j \right) =   \sum_{j=1}^{n} \lambda_j e_j =x,$ as desired.
 \end{proof}

\begin{remark}\label{remark surjective real linear isometries satisfying additional hypothesis}{\rm Let $T:C_1(H) \to C_1(H)$ be a surjective real linear isometry, where $H$ is a complex Hilbert space. Since $T^*: B(H)\to B(H)$ is a surjective real linear isometry, $T^*$ and $T$ must be complex linear or conjugate linear (see \cite[Proposition 2.6]{Da}). We therefore deduce from \cite{Ru69} (see also \cite[\S 11.2]{FleJam08}) that there exist unitary matrices $u,v\in B(H)$ such that  one of the next statements holds:
\begin{enumerate}\item[$(a.1)$] $T(x) = u x v$, for every $x\in S(C_1(H))$;
\item[$(a.2)$] $T(x) = u x^t v$, for every $x\in S(C_1(H))$;
\item[$(a.3)$] $T(x) = u \overline{x} v$, for every $x\in S(C_1(H))$;
\item[$(a.4)$] $T(x) = u x^* v$, for every $x\in S(C_1(H))$,
\end{enumerate} where $\overline{(x_{ij})} = (\overline{x_{ij}})$.\smallskip

We can find a more concrete description under additional hypothesis. Suppose dim$(H)=n$. The symbol $e_{ij}\in C_1(H)$ will denote the elementary matrix with entry $1$ at position $(i,j)$ and zero otherwise.\smallskip


\noindent\emph{\textbf{Case A}} Suppose that $T(\zeta e_{kk}) = \zeta e_{kk}$ for every $k=1,\ldots,n-1$, and all $\zeta\in\mathbb{T}$, $T(e_{n1}) = \alpha e_{n1},$ and $T(e_{1n}) = \mu e_{1n}$ with $\alpha, \mu\in \mathbb{T}$. Then $T$ has the form described in case $(a.1)$ above with $$u = \left(
\begin{array}{cccc}
1 & \ldots & 0 & 0 \\
\vdots & \ \ & \vdots & \vdots \\
0& \ldots & 1 & 0\\
0& \ldots & 0 & {\alpha}\\
 \end{array}
 \right), \hbox{ and } v = \left(
\begin{array}{cccc}
1 & \ldots & 0 & 0 \\
\vdots & \ \ & \vdots & \vdots \\
0& \ldots & 1 & 0\\
0& \ldots & 0 & {\mu}\\
 \end{array}
\right).$$ If we also assume $T(e_{nn}) = e_{nn},$ then $\alpha = \overline{\mu}$.\smallskip\smallskip

For the proof we simply observe that since $T(i e_{11}) = i e_{11}$, we discard cases $(a.3)$ and $(a.4)$. The assumption $T(e_{1n}) = \mu e_{1n}$ shows that case $(a.2)$ is impossible too. Since $T$ has the form described in $(a.1)$ for suitable $u,v$. Now, $T(e_{kk}) = e_{kk}$ implies that $u_{kk} v_{kk}=1$, for all $k=1,\ldots,n-1$. Finally, it is straightforward to check that $T(e_{1n}) = \mu e_{1n}$ and $T(e_{n1}) = \alpha e_{n1}$ give the desired statement.

%

%
%
}\end{remark}

We can present now the main result of this section.

\begin{theorem}\label{t Tingley trace class finite dim} Let $f: S(C_1(H))\to S(C_1(H))$ be a surjective isometry, where $H$ is a finite dimensional complex Hilbert space. Then there exists a surjective complex linear or conjugate linear isometry $T :C_1(H)\to C_1(H)$  satisfying $f(x) = T(x)$ for every $x\in S(C_1(H))$. More concretely, there exist unitary elements $u,v\in M_n(\mathbb{C}) = B(H)$ such that one of the following statements holds:\begin{enumerate}[$(a)$] \item $f(x) = u x v$, for every $x\in S(C_1(H))$;
\item $f(x) = u x^t v$, for every $x\in S(C_1(H))$;
\item $f(x) = u \overline{x} v$, for every $x\in S(C_1(H))$;
\item $f(x) = u x^* v$, for every $x\in S(C_1(H))$,
\end{enumerate} where $\overline{(x_{ij})} = (\overline{x_{ij}})$.
\end{theorem}

\begin{proof} We shall argue by induction on $n=$dim$(H)$. The case $n=2$ has been proved in Theorem \ref{t Tingley trace class dim 4}. Let us assume that the desired conclusion is true for every surjective isometry $f: S(C_1(K))\to S(C_1(K))$, where $K$ is a finite dimensional complex Hilbert space of dimension $\leq n$. Let $f: S(C_1(H))\to S(C_1(H))$ be a surjective isometry, where $H$ is an $(n+1)$-dimensional complex Hilbert space.\smallskip

Let $\{\xi_1,\ldots,\xi_{n+1}\}$ be an orthonormal basis of $H$, and let $e_{ij} = \xi_{j}\otimes \xi_i$. Clearly $e_{ij}$ is a pure atom for every $i,j$, and $\{e_{11},\ldots,e_{(n+1)(n+1)}\}$ is a maximal set of mutually orthogonal pure states in $C_1(H)$. By Proposition \ref{p minimal in arbitrary dimension} and Lemma \ref{l preservation of orthogonality} $\{f(e_{11}),\ldots,f(e_{(n+1)(n+1)})\}$ is a maximal set of mutually orthogonal pure atoms in $C_1(H)$ too. We can find unitary matrices $u_1, w_1\in M_{n+1} (\mathbb{C})$ such that $u_1 f(e_{ii}) w_1= e_{ii}$ for every $i= 1,\ldots, n+1$. We set $f_1 = u_1 f w_1$. We observe that $f$ admits an extension to a surjective real linear isometry if and only if $f_1$ does.\smallskip

Let us note that $\{e_{(n+1)(n+1)}\}^{\perp} := \{ x\in C_1(H) : x\perp  e_{(n+1)(n+1)}\}\cong C_1 (K)$ for a suitable $n$-dimensional complex Hilbert subspace of $H$. We regard $C_1(K)$ as a complemented subspace of $C_1(H)$ under the appropriate identification. Lemma \ref{l preservation of orthogonality} implies that $$f_1(S(C_1(K))) = f_1\left(\{e_{(n+1)(n+1)}\}^{\perp}\cap S(C_1(H)) \right) = \{f_1(e_{(n+1)(n+1)})\}^{\perp}\cap S(C_1(H))$$ $$= \{e_{(n+1)(n+1)}\}^{\perp}\cap S(C_1(H))= S(C_1(K)),$$ and hence $f_1|_{S(C_1(K))} : S(C_1(K))\to S(C_1(K))$ is a surjective isometry. By the induction hypothesis, there exist unitaries $u_n,v_n\in M_n(\mathbb{C}) = B(K)$ such that one of the following statements holds:\begin{enumerate}[$(1)$] \item $f_1(x) = u_n x v_n$, for every $x\in S(C_1(K))$;
\item $f_1(x) = u_n x^t v_n$, for every $x\in S(C_1(K))$;
\item $f_1(x) = u_n \overline{x} v_n$, for every $x\in S(C_1(K))$;
\item $f_1(x) = u_n x^* v_n$, for every $x\in S(C_1(K))$.
\end{enumerate} Let $u_{n+1} = \left(
                                 \begin{array}{cc}
                                   u_n & 0 \\
                                   0 & 1 \\
                                 \end{array}
                               \right)$ and $v_{n+1} = \left(
                                 \begin{array}{cc}
                                   v_n & 0 \\
                                   0 & 1 \\
                                 \end{array}
                               \right)$. In each case from $(1)$ to $(4)$, we can define via the unitaries $u_{n+1},v_{n+1}$ in $B(H) = M_{n+1} (\mathbb{C})$, the involution $^*$, the transposition and the conjugation $\overline{\, \cdot \, }$, a surjective complex linear or conjugate linear isometry $T_1 :C_1(H)\to C_1(H)$ such that $T_1 f_1 (x) = x,$ for every $x\in S(C_1(K))$ and $T_1 f_1(e_{(n+1)(n+1)}) = e_{(n+1)(n+1)}$.\smallskip

We deal now with the mapping $f_2 = T_1 f_1,$ which is a surjective isometry from $S(C_1(H))$ onto itself and satisfies \begin{equation}\label{eq new 14Feb 1} f_2 (x) = x,\hbox{ for every }x\in S(C_1(K))= \{e_{(n+1)(n+1)}\}^{\perp}\cap S(C_1(H)),
  \end{equation}  and $f_2(e_{(n+1)(n+1)}) = e_{(n+1)(n+1)}.$\smallskip

We claim that \begin{equation}\label{eq claim 1 0604} f_2 (e_{1(n+1)}) = \mu e_{1(n+1)}, \hbox{ and } f_2 (e_{(n+1)1}) = \overline{\mu} e_{(n+1)1},
 \end{equation} for a suitable $\mu$ in $\mathbb{T}$. Indeed, since $e_{1(n+1)}\perp e_{22},\ldots,e_{nn},e_{21}$ and $e_{(n+1)1}\perp e_{22},\ldots,e_{nn},e_{12}$, Lemma \ref{l preservation of orthogonality} implies that $$f_2 (e_{1(n+1)}) \perp f_2(e_{22})=e_{22}, \ldots, f_2 (e_{1(n+1)}) \perp  e_{nn}, f_2 (e_{1(n+1)}) \perp f_2(e_{21})=e_{21},$$ and
$$f_2 (e_{(n+1)1}) \perp f_2(e_{22})=e_{22}, \ldots, f_2 (e_{1(n+1)}) \perp e_{nn}, f_2 (e_{1(n+1)}) \perp f_2(e_{12})=e_{12},$$  which implies that $$f_2(e_{1(n+1)}) = \mu  e_{1(n+1)} + \lambda e_{(n+1)(n+1)}$$ and $$f_2(e_{(n+1)1}) = \alpha  e_{1(n+1)} + \beta e_{(n+1)(n+1)}$$ with $|\alpha|^2 + |\beta|^2 =1$ and $|\lambda|^2 + |\mu|^2 =1$ (compare Proposition \ref{p minimal in arbitrary dimension}$(b)$). Since $e_{(n+1)1}\perp e_{1(n+1)}$, a new application of Lemma \ref{l preservation of orthogonality} proves that $f_2 (e_{1(n+1)}) \perp f_2 (e_{(n+1)1})$, and consequently $\lambda=\beta =0$. We have thus proved that $$f_2 (e_{1(n+1)}) = \mu e_{1(n+1)}, \hbox{ and } f_2 (e_{(n+1)1}) = \alpha e_{(n+1)1},$$ with $\mu,\alpha$ in $\mathbb{T}$. We shall show next that $\alpha = \overline{\mu}$. To this end, we observe that the subspace $$\{e_{22},\ldots,e_{nn}\}^{\perp} := \{ x\in C_1(H) : x\perp  e_{jj}, \ \forall j\in \{2,\ldots,n\}\}$$ is isometrically isomorphic to $C_1 (K_2)$ for a suitable $2$-dimensional complex Hilbert subspace $K_2$ of $H$, contains $e_{11},$ $e_{1(n+1)},$ $e_{(n+1)1},$ and $e_{(n+1)(n+1)}$, and by Lemma \ref{l preservation of orthogonality}, $f_2|_{S(C_1(K_2))} : S(C_1(K_2))\to S(C_1(K_2))$ is a surjective isometry. So, by the induction hypothesis, there exist unitaries $u_3,v_3\in B(K_2)$ such that $f_2|_{S(C_1(K_2))}$ satisfies one of the statements from $(a)$ to $(d)$ in our theorem. Having in mind that $f_2( \zeta e_{11}) = \zeta e_{11}$ for every $\zeta\in \mathbb{T}$, $f_2(e_{(n+1)(n+1)})= e_{(n+1)(n+1)}$, $f_2 (e_{1(n+1)}) = \mu e_{1(n+1)}$  and $f_2 (e_{(n+1)1}) = \alpha e_{(n+1)1}$, it can be easily seen that we can identify $u_3$ and $v_3$ with
$\left(
\begin{array}{cc}
1 & 0 \\
0 & u_{33} \\
\end{array}
\right)$ and $\left(
\begin{array}{cc}
1 & 0 \\
0 & \overline{u_{33}} \\
\end{array}
\right)$, where $u_{33}\in \mathbb{T}$, respectively (see Remark \ref{remark surjective real linear isometries satisfying additional hypothesis} \emph{\textbf{Case A}}). This shows that $\mu = \overline{u_{33}}$ and $\alpha = {u_{33}},$ which finishes the proof of \eqref{eq claim 1 0604}.\smallskip

Now, the subspace $\{e_{(n+1)1}\}^\perp \subset  C_1(H)$ is isometrically isomorphic to $C_1 (K_3)$ for a suitable $n$-dimensional complex Hilbert space $K_3$, and since $f_2 (e_{(n+1)1}) = \overline{\mu} e_{(n+1)1}$ (see \eqref{eq claim 1 0604}), Lemma \ref{l preservation of orthogonality} implies that $$f_2|_{S(C_1(K_3))} : S(C_1(K_3))\cong M_n (\mathbb{C})\to S(C_1(K_3))\cong M_n (\mathbb{C})$$ is a surjective isometry. By the induction hypothesis there exists a surjective real linear isometry $T_3 : C_1(K_3)\to C_1(K_3)$ such that $f_2|_{S(C_1(K_3))} \equiv T_3 |_{S(C_1(K_3))}$. Since $T_3 (\zeta e_{ij}) = f_2 (\zeta e_{ij}) = \zeta e_{ij}$ for every $i,j\in\{2,\ldots,n\}$ and all $\zeta \in \mathbb{T}$ (see \eqref{eq new 14Feb 1}) and  $T_3 (e_{1 (n+1)}) =f_2 (e_{1 (n+1)}) ={\mu} e_{1(n+1)}$, Remark \ref{remark surjective real linear isometries satisfying additional hypothesis} \emph{\textbf{Case A}} shows that $T_3 (x) = x v_3$, where $v_3$ identifies with the matrix $\left(
\begin{array}{cccc}
1 & \ldots & 0 & 0 \\
\vdots & \ \ & \vdots & \vdots \\
0& \ldots & 1 & 0\\
0& \ldots & 0 & {\mu}\\
 \end{array}
 \right).$
This implies that \begin{equation}\label{eq last colum} f_2 (z) = T_3 (z) = \mu z,
\end{equation} for every $z\in S(C_1(K_3))$ with $z = z p_{n+1}$, where $p_{n+1}$ is the rank one projection $\xi_{n+1}\otimes\xi_{n+1}\in B(H)$, that is, for every $z\in S(C_1(H))$ with $z\perp e_{(n+1)1}$ and $z = z p_{n+1}$.\smallskip

Similar arguments prove that \begin{equation}\label{eq last row} f_2 (z) = \overline{\mu} z,
\end{equation} for every $z\in S(C_1(H))$ with $z\perp e_{1(n+1)}$ and $z = p_{n+1} z$.\smallskip

Let us consider the unitaries $u_4 = \left(
\begin{array}{cccc}
1 & \ldots & 0 & 0 \\
\vdots & \ \ & \vdots & \vdots \\
0& \ldots & 1 & 0\\
0& \ldots & 0 & {\mu}\\
 \end{array}
 \right)\in B(H),$ and $v_4 = \left(
\begin{array}{cccc}
1 & \ldots & 0 & 0 \\
\vdots & \ \ & \vdots & \vdots \\
0& \ldots & 1 & 0\\
0& \ldots & 0 & \overline{\mu}\\
 \end{array}
\right)\in B(H),$ and the surjective complex linear isometry defined by $T_4 (x) = u_4 x v_4$. Let $f_3: S(C_1(H))\to S(C_1(H))$ be the surjective isometry defined by $f_3 (x) = T_4 (f_2(x))$ ($x\in S(C_1(H))$). Since $T_4 (y) = y$ for every $y\in \{e_{(n+1)(n+1)}\}^{\perp}$, $T_4(e_{(n+1)(n+1)})= e_{(n+1)(n+1)}$, $T_4 (z) = \overline{\mu} z,$ for all $z\in S(C_1(H))$ with $z\perp e_{(n+1)1}$ and $z = z p_{n+1}$, and $T_4 (z) = {\mu} z,$ for all $z\in S(C_1(H))$ with $z\perp e_{1(n+1)}$ and $z = p_{n+1} z$, we deduce from \eqref{eq new 14Feb 1}, \eqref{eq last colum}, and \eqref{eq last row} that \begin{equation}\label{eq conditions for f_3} f_3 (e_{(n+1)(n+1)}) = e_{(n+1)(n+1)}, \hbox{ and } f_3 (x) = x,
\end{equation} for all $x$ in the intersection of $S(C_1(H))$ with the set $$ \{e_{(n+1)(n+1)}\}^{\perp} \cup \{z\perp e_{(n+1)1} \& z = z p_{n+1} \}\cup \{z\perp e_{1(n+1)} \& z = p_{n+1} z\}.$$\smallskip

We shall show next that \begin{equation}\label{identity on pure atoms} f_3 (v) = v,
 \end{equation} for every pure atom $v\in S(C_1 (H)).$ Let $v =\eta\otimes \zeta$ be a pure atom in $C_1(H)$. We can always write $\zeta = \lambda_1 \zeta_1 + \lambda_2 \xi_{n+1}$ and $\eta = \mu_1 \eta_1 + \mu_2 \xi_{n+1}$, with $\lambda_j,\mu_j\in\mathbb{C},$ $|\mu_1|^2 + |\mu_2|^2=1,$ $|\lambda_1|^2 + |\lambda_2|^2=1,$ and $\eta_1,\zeta_1$ are norm one elements in $\{\xi_{n+1}\}^{\perp}$. As before, we can write $$v=\alpha v_{11}+ \beta v_{12}+ \delta v_{22}+\gamma v_{21},$$ where $v_{11}=\eta_1\otimes \zeta_1$, $v_{12} = {\eta}_1 \otimes {\xi}_{n+1}$, $v_{21} = {\xi}_{n+1} \otimes {\zeta}_1$, $v_{22} = {\xi}_{n+1} \otimes {\xi}_{n+1}$, $|\alpha|^2+ |\beta|^2+|\gamma|^2+ |\delta|^2=1$, and $\alpha \delta= \beta \gamma$.\smallskip

Let us note that $\mathbb{T} v_{11}\subseteq \{e_{(n+1)(n+1)}\}^{\perp}$, $\mathbb{T} v_{12}\subseteq  \{z\perp e_{(n+1)1} \& z = z p_{n+1} \}\cap S(C_1(H))$, and $\mathbb{T} v_{12}\subseteq \{z\perp e_{1(n+1)} \& z = p_{n+1} z\}\cap S(C_1(H))$, it follows from \eqref{eq conditions for f_3} that \begin{equation}\label{eq comple linearity f_3} f_3 (\mu v_{ij}) = \mu f_3 (v_{ij})\ \hbox{ for all } (i,j)\in\{(1,1),(1,2),(2,1)\},\hbox{ and } |\mu|=1.
\end{equation}

Let us consider the space $v\in \{v_{11},v_{12},v_{21},v_{22} \}^{\perp\perp}= \hbox{Span} \{v_{11},v_{12},v_{21},v_{22} \} \cong M_2(\mathbb{C})$ whose unit sphere is fixed by $f_3$ (compare Lemma \ref{l preservation of orthogonality}). It follows from the induction hypothesis (i.e. $f_3$ satisfies one of the statements from $(a)$ to $(d)$ in the statement of the theorem for every element in the unit sphere of $\{v_{11},v_{12},v_{21},v_{22} \}^{\perp\perp}$) and \eqref{eq comple linearity f_3} that $$f_3(v) = v =\alpha v_{11}+ \beta v_{12}+ \delta v_{22}+\gamma v_{21},$$ which finishes the proof of \eqref{identity on pure atoms}.\smallskip

We shall next prove that $f_3$ satisfies the hypothesis of the above Proposition \ref{p surjective isometry identity on non maximal elements}. Let $\{ v_1, v_2, \ldots, v_k\}$ be a set of mutually orthogonal pure atoms in $S(C_1(H))$ with $k< n+1$, and real numbers $\lambda_1,\lambda_2, \ldots, \lambda_k$ satisfying $\displaystyle \sum_{j=1}^{k} |\lambda_j|=1$. Since dim$(H)=n+1$, we can find a non-empty finite set of pure atoms $\{v_{k+1},\ldots, v_{n+1}\}$ such that $\{v_{k+1},\ldots, v_{n+1}\}^{\perp} = \{v_{1},\ldots, v_{k}\}$. By \eqref{identity on pure atoms} $f_3(v_j) =v_j$ for every $j$, and then Lemma \ref{l preservation of orthogonality} assures that $f_3\left( \{v_{k+1},\ldots, v_{n+1}\}^{\perp} \cap S(C_1(H))\right) = \{v_{k+1},\ldots, v_{n+1}\}^{\perp} \cap S(C_1(H)).$ Having in mind that $\{v_{k+1},\ldots, v_{n+1}\}^{\perp} \cong C_1(H')$, where $H'$ is a complex Hilbert space with dim$(H') = k<n+1$, and $f_3|_{S(C_1(H'))} : S(C_1(H'))\to S(C_1(H'))$, it follows from the induction hypothesis the existence of a surjective real linear isometry $R: C_1(H')\to C_1(H')$ such that $f_3(x) = R(x)$ for all $x\in S(C_1(H'))$. Applying \eqref{identity on pure atoms} we get $$ f_3\left( \sum_{j=1}^{n} \mu_j v_j \right)= R\left( \sum_{j=1}^{n} \mu_j v_j \right) =  \sum_{j=1}^{n} \mu_j R(v_j) =  \sum_{j=1}^{n} \mu_j f_3(v_j) =  \sum_{j=1}^{n} \mu_j v_j.$$ \smallskip

Finally, since $f_3$ satisfies the hypothesis of the above Proposition \ref{p surjective isometry identity on non maximal elements}, we deduce from this result that $f_3 (x) = x$, for every $x\in S(C_1(H))$.
\end{proof}

\section{Surjective isometries between the unit spheres of two arbitrary trace class spaces}\label{sec:4}

In this section we consider the trace class operators on a complex Hilbert space $H$ of arbitrary dimension. The answers obtained in the finite dimensional case can be now applied to simplify the study.

\begin{theorem}\label{t Tingley for trace class infinite dimension}  Let $f: S(C_1(H))\to S(C_1(H))$ be a surjective isometry, where $H$ is an arbitrary complex Hilbert space. Then there exists a surjective complex linear or conjugate linear isometry $T :C_1(H)\to C_1(H)$  satisfying $f(x) = T(x)$ for every $x\in S(C_1(H))$.
\end{theorem}

\begin{proof} Let $\{\xi_k: k\in \mathcal{I}\}$ be an orthonormal basis of $H$. As before, we set $e_k := \xi_k\otimes \xi_k$. Then the set $\{ e_k : k\in \mathcal{I} \}$ is a maximal set of mutually orthogonal pure atoms in $S(C_1(H)).$ By Lemma \ref{l preservation of orthogonality} and Proposition \ref{p minimal in arbitrary dimension}$(b)$ the elements in the set $\{ f(e_k) : k\in \mathcal{I} \}$ are mutually orthogonal pure atoms in $C_1(H)$. We can therefore find orthonormal systems  $\{\eta_k: k\in \mathcal{I}\}$ and  $\{\zeta_k: k\in \mathcal{I}\}$ in $H$ such that $f(e_k) = \eta_k\otimes \zeta_k$ for every $k\in \mathcal{I}$.\smallskip

We claim that at least one of $\{\eta_k: k\in \mathcal{I}\}$ and $\{\zeta_k: k\in \mathcal{I}\}$ must be an orthonormal basis of $H$. Otherwise, we can find norm one elements $\eta_0$ and $\zeta_0$ in $H$ such that $\eta_0\perp \eta_k$ and $\zeta_0\perp \zeta_k$ (in $H$) for every $k$. Then the element $v_0 = \eta_0\otimes\zeta_0$ is a pure atom in $S(C_1(H))$ which is orthogonal to every $f(e_k)$. Applying Lemma \ref{l preservation of orthogonality} to $f^{-1}$ we deduce that $f^{-1} (v_0)\perp e_k$ for every $k\in \mathcal{I}$, which is impossible because $\{\xi_j: j\in \mathcal{I}\}$ is a basis of $H$. We can therefore assume that $\{\eta_k: k\in \mathcal{I}\}$ is an orthonormal basis of $H$.\smallskip

In a second step we shall show that $\{\zeta_k: k\in \mathcal{I}\}$ also is an orthonormal basis of $H$. If that is not the case, there exists $\zeta_0$ in $H$ such that $\zeta_0\perp \zeta_k$ (in $H$) for every $k$. Fix an index $k_0$ in $\mathcal{I}$ and set $v_0 := \eta_{k_0}\otimes \zeta_0\in \partial_e (\mathcal{B}_{C_1(H)})$. Clearly, $v_0 \perp f(e_k)$ for every $k\neq k_0$. Lemma \ref{l preservation of orthogonality} and Proposition \ref{p minimal in arbitrary dimension}$(b)$ imply that $f^{-1} (v_0)$ is a pure atom in $C_1(H)$ which is orthogonal to $e_k$ for every $k\neq k_0$. Since $\{\xi_j: j\in \mathcal{I}\}$ is a basis of $H$, we can easily see that $f^{-1} (v_0) =\lambda e_{k_0}$ for a unique $\lambda\in \mathbb{T}$. We deduce from Proposition \ref{p minimal in arbitrary dimension}$(d)$ and $(e)$ that $\eta_{k_0}\otimes \zeta_0= v_0 = f f^{-1} (v_0) =\mu f(e_{k_0}) = \eta_{k_0}\otimes \zeta_{k_0}$ with $\mu\in \mathbb{T}$, which contradicts the fact $\zeta_0\perp \zeta_k$ (in $H$) for every $k$.\smallskip

We have therefore shown that $\{\eta_k: k\in \mathcal{I}\}$ and $\{\zeta_k: k\in \mathcal{I}\}$ both are orthonormal basis of $H$.\smallskip

Let us pick two unitary elements $u_1,w_1\in B(H)$ such that $u_1 f(e_k) w_1 = e_k$ for every $k\in \mathcal{I}$. The mapping $f_2 :S(C_1(H))\to S(C_1(H)),$ $f_2 (x) = u_1 f(x) w_1$ is a surjective isometry and $f_2 (e_k ) = e_k$ for every $k\in \mathcal{I}$. Let $T_1$ denote the surjective complex linear isometry on $C_1(H)$ given by $T_1 (x) = u_1 x w_1$ ($x\in C_1(H)$).  \smallskip

Now, let $F$ be a finite subset of $\mathcal{I}$, and let $q_{F}$ denote the orthogonal projection of $H$ onto $H_{F} = \hbox{span} \{\xi_k: k\in F\}$. The set $\{e_k : k\notin F\}$ is invariant under $f_2$. Lemma \ref{l preservation of orthogonality} assures that $f_2 \left( \{e_k : k\notin F\}^{\perp} \cap S(C_1(H))\right) = \{e_k : k\notin F\}^{\perp} \cap S(C_1(H)),$ where $\{e_k : k\notin F\}^{\perp} \cap S(C_1(H)) = S(C_1 (H_F))$, and $$f_2|_{S(C_1 (H_F))} : S(C_1 (H_F)) \to S(C_1 (H_F))$$ is a surjective isometry. By Theorem \ref{t Tingley trace class finite dim} there exists a surjective real linear isometry $T_{_{F}} : C_1 (H_F) \to C_1 (H_F)$ such that $f_2 (x) = T_{_F} (x)$ for all $x\in S(C_1 (H_F))$.\smallskip

Let $T_2 : C_1 (H) \to C_1 (H)$  denote the homogeneous extension of $f_2$ defined by $T_2(x) := \|x\|_1 f_2\left( \frac{x}{\|x\|_1} \right)$ if $x\neq 0$ and $T_2(0)=0$. Clearly $T_2$ is surjective. We shall show that $T_2$ is an isometry. To this end, let us fix $x,y\in C_1(H)\backslash\{0\}$ and $\varepsilon>0$. Since every element in $C_1(H) (\subseteq K(H))$ can be approximated in norm by elements $x\in S(C_1(H))$ with $x = q_F  x q_F$, where $F$ is a finite subset of $\mathcal{I}$, we can find a finite set $F\subset \mathcal{I}$, $x_{\varepsilon}$ and $y_{\varepsilon}$ in $S(C_1(H))$ such that $x_{\varepsilon} = q_F  x_{\varepsilon} q_F$, $y_{\varepsilon} = q_F  y_{\varepsilon} q_F$, $\|\frac{x}{\|x\|_1} - x_{\varepsilon}\|_1<\frac{\varepsilon}{2(\|x\|_1+\|y\|_1)}$ and $\|\frac{y}{\|y\|_1} - y_{\varepsilon}\|_1<\frac{\varepsilon}{2(\|x\|_1+\|y\|_1)}$. By the triangular inequality we have $$\Big| \left\| x - y \right\|_1 - \left\| \|x\|_1 x_{\varepsilon} - \|y\|_1 y_\varepsilon \right\|_1 \Big|\leq \left\| x - \|x\|_1 x_{\varepsilon}  \right\|_1 + \left\| y -  \|y\|_1 y_\varepsilon \right\|_1$$ $$\leq \|x\|_1 \left\| \frac{x}{\|x\|_1} -  x_{\varepsilon}  \right\|_1 + \|y\|_1 \left\| \frac{y}{\|y\|_1} -  y_\varepsilon \right\|_1< \frac{\varepsilon}{2},$$ and since $f_2$ is an isometry we get
$$\Big| \left\| T_2(x) - T_2(y) \right\|_1 - \left\| T_2(\|x\|_1 x_{\varepsilon}) - T_2(\|y\|_1 y_\varepsilon) \right\|_1 \Big|$$ 
$$\leq \left\| T_2(x) - T_2(\|x\|_1 x_{\varepsilon})  \right\|_1 + \left\| T_2(y) -  T_2(\|y\|_1 y_\varepsilon) \right\|_1$$ $$\leq \|x\|_1 \left\| f_2\left(\frac{x}{\|x\|_1}\right) - f_2\left( x_{\varepsilon}\right)  \right\|_1 + \|y\|_1 \left\| f_2\left(\frac{y}{\|y\|_1}\right) -  f_2( y_\varepsilon) \right\|_1$$ $$= \|x\|_1 \left\| \frac{x}{\|x\|_1} -  x_{\varepsilon}  \right\|_1 + \|y\|_1 \left\| \frac{y}{\|y\|_1} -  y_\varepsilon \right\|_1< \frac{\varepsilon}{2}.$$

On the other hand, since $y_\varepsilon,x_\varepsilon\in S(C_1 (H_F))$, we can consider the surjective real linear isometry $T_{_{F}} : C_1 (H_F) \to C_1 (H_F)$ satisfying $f_2 (x) = T_{_F} (x)$ for all $x\in S(C_1 (H_F))$ to deduce that $$\left\| T_2(\|x\|_1 x_{\varepsilon}) - T_2(\|y\|_1 y_\varepsilon) \right\|_1  = \left\| \|x\|_1 f_2( x_{\varepsilon}) - \|y\|_1 f_2(y_\varepsilon) \right\|_1 $$ $$= \left\| \|x\|_1 T_{_F}( x_{\varepsilon}) - \|y\|_1  T_{_F} (y_\varepsilon) \right\|_1 = \left\| T_{_F}( \|x\|_1 x_{\varepsilon}  - \|y\|_1  y_\varepsilon) \right\|_1 = \left\| \|x\|_1 x_{\varepsilon}  - \|y\|_1  y_\varepsilon \right\|_1. $$ Combining this identity with the previous two inequalities we obtain $$ \left| \left\| T_2(x) - T_2(y) \right\|_1 - \left\| x - y \right\|_1 \right| < \varepsilon.$$ The arbitrariness of $\varepsilon$ shows that $\left\| T_2(x) - T_2(y) \right\|_1 = \left\| x - y \right\|_1$, and hence $T_2$ is an isometry.\smallskip

Finally, since $T_2$ is a surjective isometry with $T_2 (0)=0$, the Mazur-Ulam theorem guarantees that $T_2$ is a surjective real linear isometry, and hence $f(x) = T_1^{-1} T_2 (x)$ for all  $x\in S(C_1(H))$, witnessing the desired conclusion.
\end{proof}

\textbf{Acknowledgements} First, second and third author were partially supported by the Spanish Ministry of Economy and Competitiveness (MINECO) and European Regional Development Fund project no. MTM2014-58984-P and Junta de Andaluc\'{\i}a grant FQM375. Fourth author partially supported by grants MTM2014-54240-P, funded by MINECO and QUITEMAD+-CM, Reference: S2013/ICE-2801, funded by Comunidad de Madrid.\smallskip

The authors are indebted to the anonymous reviewer for a thorough report, insightful comments, and suggestions.

\end{document}